\documentclass[11pt,reqno]{amsart}
\usepackage{amssymb}
\usepackage{amsmath}
\usepackage{mathrsfs}
\usepackage{amssymb, amsmath, amsfonts}
\usepackage{bbm}
\usepackage{mathabx,url}

\usepackage{color}

\usepackage[hidelinks]{hyperref}
\usepackage{amsmath, amssymb}
\usepackage{enumitem}

\makeatletter
\def\tank#1{\protected@xdef\@thanks{\@thanks
		\protect\footnotetext[0]{#1}}}
\def\bigfoot{
	
	\@footnotetext}
\makeatother

\newcommand{\ea}{\end{array}}

\allowdisplaybreaks
\numberwithin{equation}{section}
\newtheorem{theorem}{Theorem}[section]
\newtheorem{lemma}{Lemma}[section]
\newtheorem{proposition}{Proposition}[section]

\newtheorem{remark}{Remark}[section]

\newtheorem{definition}{Definition}[section]

\def\beq{\begin{equation}}
\def\nneq{\end{equation}}

\def\bthm{\begin{theorem}}
\def\nthm{\end{theorem}}

\title[SBP and Chung's LIL for Gaussian Volterra processes]{Small ball probabilities and Chung's law of the iterated logarithm for Gaussian Volterra processes with power-type kernels}

\author[M. Tong]{Mengxin Tong}
\address[]{Mengxin  Tong, School of Mathematics and Statistics,  Wuhan University,  Wuhan, 430072,
China.}
\email{mengxintong@whu.edu.cn}

\author[R. Wang]{Ran Wang}
\address[]{Ran Wang, School of Mathematics and Statistics,  Wuhan University,  Wuhan, 430072,
China.}
\email{rwang@whu.edu.cn}

\author[Q. Yang]{Qingshan Yang}
\address[ ]{Qingshan Yang, School of Mathematics and Statistics, Northeast Normal University, Changchun,
130024, China.}\email{yangqr66@gmail.com}

\date{}
\begin{document}
\maketitle

{\bfseries Abstract:}  
 Consider the Gaussian Volterra process introduced by Mishura and Shklyar
\cite{MS22a,MS22b},
$$
X(t)
=
\int_0^t r^\alpha
\left(
\int_r^t u^\beta (u-r)^\gamma\,du
\right)dW_r,
\qquad t\ge 0,
$$
where
$$
\alpha>-\frac12,
\quad
\gamma\in\left(-1,-\frac12\right),
\quad
H:=\alpha+\beta+\gamma+\frac32>0.
$$
We obtain  two-sided estimates for the small ball probabilities of
$X$. As applications, we prove Chung's laws of the iterated logarithm (Chung's LILs) at
every fixed point $t>0$, at the origin, and at infinity.  The fixed-time result follows from the small ball estimates and the Lamperti transformation, whereas the results at the origin and infinity follow from Talagrand’s lower-class criteria \cite{talagrand1996lower}.    
These results show that $\gamma+\frac32$ determines the local roughness and the small ball exponent, $\alpha+\beta$ determines the scale of local fluctuations at fixed positive times, and $H$ governs the self-similar scaling at the origin and infinity.



\section{Introduction}
 The Gaussian Volterra process (GVP)
$X=\{X(t),t\geq 0\}$, introduced by Mishura and Shklyar \cite{MS22a},
is a centered Gaussian process defined by
\begin{equation}\label{eq GVP}
	X(t)
	=
	\int_0^t r^\alpha
	\left(
		\int_r^t u^\beta (u-r)^\gamma\,du
	\right)dW_r,
	\qquad t\geq 0,
\end{equation}
where $W=\{W(t),t\geq 0\}$ is a standard Brownian motion and the
parameters satisfy
\begin{equation}\label{eq constant}
	\alpha>-\frac12,
	\qquad
	\gamma>-1,
	\qquad
	\alpha+\beta+\gamma>-\frac32.
\end{equation} Under \eqref{eq constant}, the process $X$ is well defined and
self-similar with Hurst index
\begin{equation}\label{eq H}
H:=\alpha+\beta+\gamma+\frac32>0;
\end{equation}
see Mishura and Shklyar
\cite[Theorem~1 and Proposition~1]{MS22a}.

 For $H\in(1/2,1)$, if we choose
$
\alpha=\frac12-H,\,
\beta=H-\frac12,\,
\gamma=H-\frac32,
$
then 
the process $X$ reduces to
$$
\int_0^t r^{\frac12-H}
\left(
\int_r^t u^{H-\frac12}(u-r)^{H-\frac32}\,du
\right)dW_r.
$$
Up to a multiplicative constant, this representation coincides with the
Molchan representation (also known as the Volterra-type representation) of
fractional Brownian motion (fBm) $B^H$ with Hurst index $H$.
See, for example, \cite[Section~1.1]{BMRS},
\cite[Theorem~1.8.3]{Mis08}, and
\cite[Theorem~5.2]{NVV99}.

Another important special case arises when $\beta=0$. If
$
\alpha\in\left(-1/2,0\right],
\gamma\in\left(-1-\alpha,-1/2-\alpha\right),
$
then $X$ becomes
\begin{align*}\label{eq GRL}
X(t)
=
\frac{1}{1+\gamma}
\int_0^t (t-s)^{\gamma+1}s^\alpha\,dW_s .
\end{align*}
This process is   the generalized Riemann--Liouville fBm studied by Ichiba, Pang, and Taqqu
\cite[(2.13)]{TGM2022}.
 
In \cite{MS22a}, Mishura and Shklyar established several fundamental
properties of the GVP, including asymptotic estimates for incremental
variances, the generalized quasi-helix property, and H\"older continuity.
In a subsequent work \cite{MS22b}, they further investigated path
differentiability and obtained an inverse Volterra representation for the
process. More recently, El Omari \cite{EO24} and  Mishura,   Ralchenko,  and Shklyar \cite{MRS23}
studied several structural
properties of the GVP, including its non-Markovian behavior,
semimartingale property, long-range dependence, and parametric smoothness.
However, precise sample path properties of the GVP, such as small ball
probabilities and laws of the iterated logarithm, remain largely unexplored.
 We therefore establish sharp small ball probability estimates, up to multiplicative constants in the exponent, and derive Chung’s LILs for the GVP. The main technical contributions underlying these results are described below.

 Theorems \ref{lem small0t} and \ref{lem smallt-rt+r} provide, respectively, global and local small ball estimates and reveal their distinct scaling behaviors.
\begin{theorem}\label{lem small0t}
Assume that
\begin{equation}\label{eq constant cond}
\begin{aligned}
	\alpha>-\frac12,\qquad
	\gamma\in\left(-1,-\frac12\right),\qquad
	H=\alpha+\beta+\gamma+\frac32>0.
\end{aligned}
\end{equation}
Then there exist constants $\kappa_1,\kappa_2\in(0,\infty)$ such that, for
every $t>0$ and $0<\varepsilon<t^H$,
\begin{equation}\label{Eq small0t}
\exp\left(
-\kappa_1
\left(\frac{t^H}{\varepsilon}\right)^{\frac{2}{2\gamma+3}}
\right)
\leq
\mathbb{P}
\left\{
\sup_{s\in[0,t]}|X(s)|\leq\varepsilon
\right\}
\leq
\exp\left(
-\kappa_2
\left(\frac{t^H}{\varepsilon}\right)^{\frac{2}{2\gamma+3}}
\right).
\end{equation}
\end{theorem}

 \begin{theorem}\label{lem smallt-rt+r} Under Condition \eqref{eq constant cond}, there exist constants
\(\kappa_3,\kappa_4\in(0,\infty)\) such that, 
for all \(t>0\), \(r\in(0,t/2)\), and
\(0<\varepsilon<t^{\alpha+\beta}r^{\gamma+\frac32}\),
\begin{equation}\label{Eq smallt-rt+r}
\begin{split}
\exp\left(
-\kappa_3 r t^{\frac{2\alpha+2\beta}{2\gamma+3}}
\varepsilon^{-\frac{2}{2\gamma+3}}
\right)
&\leq
\mathbb{P}
\left\{
\sup_{|s|\le r}|X(t+s)-X(t)|\leq\varepsilon
\right\}\\
&\leq
\exp\left(
-\kappa_4 r t^{\frac{2\alpha+2\beta}{2\gamma+3}}
\varepsilon^{-\frac{2}{2\gamma+3}}
\right).
\end{split}
\end{equation}
\end{theorem}

 Small ball probability estimates are useful not only in the proof of
Chung’s LIL, but also in
the study of various fine sample path properties of Gaussian processes. We
refer the reader to Li and Shao \cite{li_shao_2001} for further applications.

\begin{remark}
When $\gamma>-1/2$, Mishura and Shklyar~\cite[Corollary~2]{MS22b}
showed that $X$ admits continuously differentiable sample paths on $(0,\infty)$. Related sample-path properties of the derivative process $X'$ may be studied by arguments similar to those in Wang and Xiao  \cite{WX2022a}.

The critical case $\gamma=-1/2$ requires different techniques and is excluded here;
see  Wang and Xiao \cite[Remarks~1.1 and~5.1]{WX2022a} for a related discussion. Henceforth, we assume $\gamma\in(-1,-1/2)$.
\end{remark}

\begin{theorem}\label{thm CLIL}
Under Condition \eqref{eq constant cond}, there exists a constant
$\kappa_5\in(0,\infty)$ such that, for every $t>0$,
\begin{align}\label{eq CLIL1}
\liminf_{r\to0^+}
\sup_{|s|\le r}
\frac{ (\log\log r^{-1})^{\gamma+\frac32} }
{r^{\gamma+\frac32} }|X(t+s)-X(t)|
=
\kappa_5 t^{\alpha+\beta},
\qquad \text{a.s.}
\end{align}
\end{theorem}

We prove Theorem~\ref{thm CLIL} by combining the small ball probability estimates with the Lamperti transformation, following the approaches of Tudor and Xiao \cite[Theorem~3.1]{TX2007} and Wang and Xiao \cite[Proposition~6.1]{WX2022a}. This method is well suited to the local problem at a fixed positive time, but it does not apply directly at the origin or at infinity. These two regimes are therefore treated separately in Theorem~\ref{thm CLTL0infinity}, whose proof is based on Talagrand-type lower-class criteria.

\begin{theorem}\label{thm CLTL0infinity}
Under Condition \eqref{eq constant cond}, there exist  constants
$\kappa_6,\kappa_7\in(0,\infty)$ such that
\begin{equation*}\label{thm CLTL0}
\liminf_{t\to0^+}
\sup_{0\le s\le t}
\frac{(\log\log t^{-1})^{\gamma+\frac32}}
{t^H}|X(s)|
=
\kappa_6,
\qquad\text{a.s.,}
\end{equation*}
and
\begin{equation*}\label{thm CLTLinfinity}
\liminf_{t\to\infty}
\sup_{0\le s\le t}
\frac{(\log\log t)^{\gamma+\frac32}}
{t^H }|X(s)|
=
\kappa_7,
\qquad\text{a.s.}
\end{equation*}
\end{theorem}

Theorems 1.3 and 1.4 establish Chung's LILs for $X$ in different regimes. These results show that the parameters enter the asymptotic behavior through different combinations.  
In particular, $\gamma+\frac32$ determines the local roughness, $\alpha+\beta$ determines the scale of local fluctuations at fixed positive times, and $H$ governs the self-similar scaling at the origin and infinity.

The rest of the paper is organized as follows.  
In Section 2, we establish moment estimates for the increments of $X$ and prove the one-sided strong local nondeterminism (SLND) property of $X$.  In Section~3, we  prove Theorems~\ref{lem small0t} and
\ref{lem smallt-rt+r}.  
In Section 4, we prove Theorem 1.3. In Section 5, we establish the lower-class criteria used to prove Theorem 1.4.

\section{Moment estimates for increments and one-sided SLND}
 In this section, we establish two-sided moment estimates for the increments
of $X$ and prove the  one-sided SLND
property.  

\subsection{Moment estimates}

\begin{lemma}\label{lem mom}
Assume that Condition \eqref{eq constant cond} holds. Then there exist
positive constants $c_{2,1}$ and $c_{2,2}$ such that, for all $t>s>0$,
\begin{equation}\label{eq moment}
\begin{aligned}
c_{2,1}t^{2\alpha+2\beta}(t-s)^{2\gamma+3}
&\leq
\mathbb{E}\left[(X(t)-X(s))^2\right] \\
&\leq
c_{2,2}
\max\left\{
t^{2\alpha+2\beta},
s^{2\alpha+2\beta}
\right\}
(t-s)^{2\gamma+3}.
\end{aligned}
\end{equation}
\end{lemma}

\begin{proof}
Following the decomposition in \cite[Equation (8)]{MS22a}, for any
$t>s>0$, we have
\begin{equation}\label{xt-xs}
\begin{aligned}
X(t)-X(s)
={}&
\int_0^s r^{\alpha}
\left(
\int_s^t u^\beta(u-r)^\gamma\,du
\right)dW_r\\
&+
\int_s^t r^{\alpha}
\left(
\int_r^t u^\beta(u-r)^\gamma\,du
\right)dW_r .
\end{aligned}
\end{equation}
By Itô's isometry and the orthogonality of stochastic integrals over
disjoint intervals, we obtain
\begin{equation}\label{i12}
\begin{aligned}
\mathbb{E}\left[(X(t)-X(s))^2\right]
={}&
\int_0^s r^{2\alpha}
\left(
\int_s^t u^\beta(u-r)^\gamma\,du
\right)^2dr\\
&+
\int_s^t r^{2\alpha}
\left(
\int_r^t u^\beta(u-r)^\gamma\,du
\right)^2dr\\
=:&\, I_1+I_2 .
\end{aligned}
\end{equation}
	
	 \noindent\textbf{Step 1. Lower bound.}
Since $I_1\ge0$, it suffices to establish a lower bound for $I_2$.

We first consider the case $\beta>0$. Since $u\mapsto u^\beta$ is increasing on $(0,\infty)$, we have
\begin{equation}\label{eq}
\begin{aligned}
I_2
&\ge
\int_s^t r^{2\alpha}
\left(
r^\beta\int_r^t (u-r)^\gamma\,du
\right)^2dr\\
&\ge
\frac{1}{(\gamma+1)^2}
\int_{\frac{s+t}{2}}^t
r^{2\alpha+2\beta}(t-r)^{2\gamma+2}\,dr\\
&\ge
\frac{\min\{2^{-(2\alpha+2\beta)},1\}}
{2^{2\gamma+3}(2\gamma+3)(\gamma+1)^2}
t^{2\alpha+2\beta}(t-s)^{2\gamma+3}.
\end{aligned}
\end{equation}
In the last line, we used the elementary inequality
\begin{equation}\label{eq r0}
r^x
\ge
\min\left\{\left(t/2\right)^x,t^x\right\}
=
\min\{2^{-x},1\}t^x,
\end{equation}
which holds for every fixed $x\in\mathbb R$ and
$r\in[(s+t)/2,t]$.

The proof for $\beta\le0$ is analogous,  with $u^\beta\ge r^\beta$ replaced by $u^\beta\ge t^\beta$ and \eqref{eq r0} applied with $x=2\alpha$. We omit the details.

\noindent\textbf{Step 2. Upper bound.}
We estimate $I_1$ and $I_2$ separately.

\noindent\textbf{Estimate of $I_1$.}
Following the argument in \cite[Lemma~3.1]{WX2022a}, we consider two cases:
\[
s\le 2(t-s)\quad\text{and}\quad s>2(t-s).
\]

We first consider the case $s\leq 2(t-s)$. Using the changes of variables \(u=s+(t-s)v\) and \(r=w(t-s)\), we obtain
\begin{equation*}
I_1
=
(t-s)^{2H}
\int_0^{\frac{s}{t-s}}
w^{2\alpha}
\left[
\int_0^1
\left(
\frac{s}{t-s}+v
\right)^\beta
\left(
v+\frac{s}{t-s}-w
\right)^\gamma
\,dv
\right]^2dw .
\end{equation*}
 
	Since $\gamma\in (-1,-1/2)$,   for every  $w\in [0,  \frac{s}{t-s}]$,  we have
	\begin{equation*}
		\begin{split}
			\int_{0}^{1} \left(\frac{s}{t-s} + v\right)^{\beta} \left( v + \frac{s}{t-s} - w\right)^{\gamma} dv 	\le\, \frac{1}{\gamma+1}\max\left\{\left(\dfrac{t}{t-s}\right)^\beta, \left(\dfrac{s}{t-s}\right)^\beta\right\}.
		\end{split}
	\end{equation*} 
	Consequently, we get
\begin{equation}\label{I1upperbound2}
		\begin{split}
			I_1 \le &\,   \frac{(t-s)^{2H} \cdot \max\left\{ (\frac{t}{t-s})^{2\beta}, \, (\frac{s}{t-s})^{2\beta}\right\}}{(\gamma+1)^2}   \int_{0}^{\frac{s}{t-s}} w^{2\alpha} dw\\
			= &\, \frac{1}{(\gamma+1)^2 (2\alpha+1)}  \cdot \max\left\{  t^{2\beta}s^{2\alpha+1},\, s^{2\alpha+2\beta+1}\right\}(t-s)^{2\gamma+2}  \\
			\le &\, \frac{1}{(\gamma+1)^2 (2\alpha+1)}  \cdot \max\left\{  t^{2\beta} t^{2\alpha+1},\, 3 s^{2\alpha+2\beta} (t-s)\right\}(t-s)^{2\gamma+2}  \\
			\le &\, \frac{1}{(\gamma+1)^2 (2\alpha+1)}  \cdot \max\left\{  3t^{2\alpha+2\beta},\, 3s^{2\alpha+2\beta}\right\}(t-s)^{2\gamma+3},
		\end{split}
	\end{equation} 
where we have used $2\alpha+1>0$ and $s<t\le3(t-s)$.

We next consider the case $s>2(t-s)$, i.e., $2t/3<s<t$.
Hence,
\begin{align*}
\int_s^t u^\beta (u-r)^\gamma du
&\le
\max\{t^\beta,s^\beta\}
\int_s^t (u-r)^\gamma du\\
&\le
c_{2,3}s^\beta
\left[
(t-r)^{\gamma+1}-(s-r)^{\gamma+1}
\right],
\end{align*}
where
$c_{2,3}=\frac{1}{\gamma+1}\max\{(3/2)^\beta,1\}$.

    Using the change of variables $r = s - (t-s)v$, we have
	\begin{equation} \label{eq:I1_bound}
		\begin{split}	 
			I_1 \le  & \,  c_{2,3}^{2} s^{2\beta}\int_{0}^{s} r^{2\alpha} \bigl((t-r)^{\gamma+1}-(s-r)^{\gamma+1}\bigr)^2 dr \\
			=  &\,  c_{2,3}^{2} s^{2\beta}(t-s)^{2\alpha+2\gamma+3} \int_{0}^{\frac{s}{t-s}} \left(\frac{s}{t-s}-v\right)^{2\alpha}[(1+v)^{\gamma+1}-v^{\gamma+1}]^2  dv\\
			= &\,  c_{2,3}^{2} s^{2\beta}(t-s)^{2\alpha+2\gamma+3} \left( \int_{0}^{1} \bigl[(1+v)^{\gamma+1}-v^{\gamma+1}\bigr]^2 \left(\frac{s}{t-s}-v\right)^{2\alpha} dv \right. \\
			&\qquad \left. + \int_{1}^{\frac{s}{t-s}} \bigl[(1+v)^{\gamma+1}-v^{\gamma+1}\bigr]^2 \left(\frac{s}{t-s}-v\right)^{2\alpha} dv \right).
		\end{split}
	\end{equation}
	Combining the interval inequality
\[
\frac{s}{2(t-s)} \leq \frac{s}{t-s} - v \leq \frac{s}{t-s},
\qquad v\in[0,1],
\]
with the elementary inequality
\[
(1+v)^{\gamma+1}-v^{\gamma+1}\le1,\qquad v\in[0,1],
\]
we obtain
	\begin{equation}\label{eq:I11_bound}
		\int_{0}^{1} \bigl[(1+v)^{\gamma+1}-v^{\gamma+1}\bigr]^2 \left(\frac{s}{t-s}-v\right)^{2\alpha} dv \leq \max\left\{ 1,\;  2^{-2\alpha}\right\} \left( \frac{s}{t-s}\right)^{2\alpha}. 
	\end{equation}
Furthermore, by the elementary inequality
\begin{equation*}
(1+v)^{\gamma+1}-v^{\gamma+1}
\leq
(\gamma+1)v^\gamma,
\qquad v\ge1,
\label{eq:basic-ineq}
\end{equation*}
we obtain
\begin{equation*}
		\begin{split}
&\int_{1}^{\frac{s}{t-s}}
\left[(1+v)^{\gamma+1}-v^{\gamma+1}\right]^2
\left(\frac{s}{t-s}-v\right)^{2\alpha}dv
\nonumber\\
&\le
(\gamma+1)^2
\int_{1}^{\frac{s}{t-s}}
v^{2\gamma}
\left(\frac{s}{t-s}-v\right)^{2\alpha}dv
\nonumber\\
&=
(\gamma+1)^2
\left(\frac{s}{t-s}\right)^{2\alpha+2\gamma+1}
\int_{\frac{t-s}{s}}^1
w^{2\gamma}(1-w)^{2\alpha}dw \nonumber \\
&=
(\gamma+1)^2
\left(\frac{s}{t-s}\right)^{2\alpha+2\gamma+1}
\left(
\int_{\frac{t-s}{s}}^{\frac23}
w^{2\gamma}(1-w)^{2\alpha}dw
+
\int_{\frac23}^{1}
w^{2\gamma}(1-w)^{2\alpha}dw
\right).
		\end{split}
\end{equation*}
Using $1/3\le 1-w\le 1$ and $w\ge 2/3$ in the first and second integrals, respectively, there exist constants $c_{2,4},c_{2,5}>0$ such that
\[
\int_{\frac{t-s}{s}}^1
w^{2\gamma}(1-w)^{2\alpha}\,dw
\le
c_{2,4}
\left(\frac{t-s}{s}\right)^{2\gamma+1}
+c_{2,5}.
\]
Since \(\left[s/(t-s)\right]^{2\gamma+1}\le 2^{2\gamma+1}\), there exists a constant \(c_{2,6}>0\) such that
\begin{align}\label{eq:I12_bound}
\int_{1}^{\frac{s}{t-s}}
\left[(1+v)^{\gamma+1}-v^{\gamma+1}\right]^2
\left(\frac{s}{t-s}-v\right)^{2\alpha}dv
\le
c_{2,6}
\left(\frac{s}{t-s}\right)^{2\alpha}.
\end{align} 
From \eqref{eq:I1_bound}, \eqref{eq:I11_bound}, and \eqref{eq:I12_bound}, there exists a constant \(c_{2,7}>0\) such that
\begin{equation}\label{I1upperbound3}
I_1\leq c_{2,7}s^{2\alpha+2\beta}(t-s)^{2\gamma+3}.
\end{equation}
Thus,   \eqref{I1upperbound2} and
\eqref{I1upperbound3} give the required bound for $I_1$.

\noindent\textbf{Estimate of $I_2$.} 
We distinguish the two cases $\beta>0$ and $\beta\le0$. 

For $\beta>0$, we again consider $s\le2(t-s)$ and $s>2(t-s)$. When $s\le2(t-s)$, since $u\mapsto u^\beta$ is increasing on $(0,\infty)$, we have
\begin{equation}\label{I2upperbound1}
\begin{aligned}
I_2
&\leq
\frac{t^{2\beta}}{(\gamma+1)^2}
\int_s^t r^{2\alpha}(t-r)^{2\gamma+2}\,dr\\
&\leq
\frac{t^{2\beta}}{(\gamma+1)^2}
\int_0^t r^{2\alpha}(t-r)^{2\gamma+2}\,dr\\
&=
\frac{B(2\alpha+1,2\gamma+3)}
{(\gamma+1)^2}
t^{2\alpha+2\beta+2\gamma+3}\\
&\leq
\frac{3^{2\gamma+3}B(2\alpha+1,2\gamma+3)}
{(\gamma+1)^2}
t^{2\alpha+2\beta}(t-s)^{2\gamma+3},
\end{aligned}
\end{equation}
where the last inequality follows from \(t\le 3(t-s)\) and \(2\gamma+3>0\).

When $s>2(t-s)$, proceeding as in the first line of \eqref{I2upperbound1}, we have
\begin{equation}\label{I2upperbound2}
\begin{aligned}
I_2
&\leq
\frac{t^{2\beta}}{(\gamma+1)^2}
\int_s^t r^{2\alpha}(t-r)^{2\gamma+2}dr\\
&\leq
\frac{t^{2\beta}}{(\gamma+1)^2}
\max\{t^{2\alpha},s^{2\alpha}\}
\int_s^t(t-r)^{2\gamma+2}dr\\
&=
\frac{1}{(\gamma+1)^2(2\gamma+3)}
\max\{t^{2\alpha+2\beta},
t^{2\beta}s^{2\alpha}\}
(t-s)^{2\gamma+3}\\
&\leq
\frac{1}{(\gamma+1)^2(2\gamma+3)}
\max\left\{
t^{2\alpha+2\beta},
\left(3/2\right)^{2\beta}s^{2\alpha+2\beta}
\right\}
(t-s)^{2\gamma+3}.
\end{aligned}
\end{equation}
Here, in the second inequality, we have used
\begin{equation}\label{eq r mono1}
r^x\leq\max\{t^x,s^x\},
\qquad r\in[s,t],
\end{equation}
which follows from the monotonicity of the function
$r\mapsto r^x$ on $[s,t]$. The last inequality follows from \(t<(3/2)s\).

For $\beta\le0$, since the function $u\mapsto u^\beta$ is decreasing on $(0,\infty)$, we have
\begin{equation}\label{I2upperbound3}
\begin{aligned}
I_2
&\leq
\int_s^t r^{2\alpha}
\left(
r^\beta\int_r^t (u-r)^\gamma du
\right)^2dr\\
&=
\frac{1}{(\gamma+1)^2}
\int_s^t r^{2\alpha+2\beta}
(t-r)^{2\gamma+2}dr\\
&\leq
\frac{1}{(\gamma+1)^2(2\gamma+3)}
\max\left\{
t^{2\alpha+2\beta},
s^{2\alpha+2\beta}
\right\}
(t-s)^{2\gamma+3},
\end{aligned}
\end{equation}
where we have used \eqref{eq r mono1} with
$x=2\alpha+2\beta$.
 
Combining these bounds with 
\eqref{i12} proves \eqref{eq moment}. The proof is complete. 
\end{proof}

\begin{remark}  By Lemma~\ref{lem mom}, for every fixed $t>0$, there exist positive constants
$c_{2,8}$ and $c_{2,9}$ such that
$$
c_{2,8}t^{2\alpha+2\beta}
\leq
\liminf_{s\uparrow t}
\frac{\mathbb{E}\bigl[(X(t)-X(s))^2\bigr]}
{(t-s)^{2\gamma+3}}
\leq
\limsup_{s\uparrow t}
\frac{\mathbb{E}\bigl[(X(t)-X(s))^2\bigr]}
{(t-s)^{2\gamma+3}}
\leq
c_{2,9}t^{2\alpha+2\beta}.
$$
The asymptotic estimate in \cite[Proposition~2]{MS22a} is a local estimate as $s\uparrow t$  and therefore does not provide the uniform bounds needed below. Lemma \ref{lem mom} supplies the uniform two-sided estimates required for the one-sided SLND and the small ball bounds.

\end{remark}
\begin{remark}
 For every compact interval $[T_1,T_2]\subset(0,\infty)$, Lemma~\ref{lem mom}
implies that there exist positive constants $c_{2,10}$ and $c_{2,11}$,
depending on $T_1$ and $T_2$, such that
$$
c_{2,10}(t-s)^{2\gamma+3}
\leq
\mathbb{E}\bigl[(X(t)-X(s))^2\bigr]
\leq
c_{2,11}(t-s)^{2\gamma+3},
\qquad T_1\leq s<t\leq T_2.
$$
Consequently, $X$ satisfies the quasi-helix property with exponent
$\gamma+\frac32$ on every compact interval bounded away from the origin,
in the sense of \cite[Definition~5.1]{MS22a}.
\end{remark}

\subsection{One-sided SLND} 
\begin{proposition}\label{lem SLND}
Assume that Condition \eqref{eq constant cond} holds. Then, for every $t> s\ge0$,
\begin{align}\label{SLND}
\operatorname{Var}
\left(
X(t)\mid \sigma\{X(r):\, 0\le r\le s\}
\right)
\geq
c_{2,1}t^{2\alpha+2\beta}|t-s|^{2\gamma+3},
\end{align}
where $\operatorname{Var}(X(t)\mid\sigma\{X(r):r\le s\})$ denotes the
conditional variance of $X(t)$ given the $\sigma$-algebra generated by
$\{X(r):r\le s\}$.
\end{proposition}

\begin{proof}
For $t>s\ge0$, we decompose $X(t)$ as
\begin{align*}
X(t)
={}
\int_0^{s} r^{\alpha}
\left(
\int_r^t u^{\beta}(u-r)^{\gamma}du
\right)dW_r
+
\int_s^t r^{\alpha}
\left(
\int_r^t u^{\beta}(u-r)^{\gamma}du
\right)dW_r .
\end{align*}

The first term is measurable with respect to
$\sigma\{W(r):r\le s\}$, while the second term is independent of this
$\sigma$-algebra. Moreover, by \cite[Theorem~2]{MS22b},
\[
\sigma\{X(r):\, 0\le r\le s\}
=
\sigma\{W(r):\, 0\le r\le s\}.
\]
Therefore,
\begin{align*}
\operatorname{Var}
\left(
X(t)\mid\sigma\{X(r):r\le s\}
\right)
&=
\operatorname{Var}
\left(
X(t)\mid\sigma\{W(r):r\le s\}
\right)
\\
&=
\int_s^t r^{2\alpha}
\left(
\int_r^t u^\beta(u-r)^\gamma du
\right)^2dr .
\end{align*} 
This is precisely the term $I_2$ in \eqref{i12}. Applying the lower bound for \(I_2\) obtained in the proof of Lemma~\ref{lem mom} gives \eqref{SLND}.

  The proof is complete. 
\end{proof}

\section{Small ball probabilities}

In this section, we prove the upper bounds in Theorems \ref{lem small0t} and~\ref{lem smallt-rt+r} by combining one-sided SLND with the methods of Monrad and Rootz\'en
\cite[Theorem~2.1]{MR1995} and Xiao \cite[Theorem~3.1]{Xiao2008}, and  prove the lower bounds via Talagrand’s principle \cite[Lemma~2.2]{talagrand1996lower} in Ledoux’s formulation \cite[pp.~257, (7.11)--(7.13)]{Ledoux}.

\subsection{Proof of Theorem \ref{lem small0t}} 

  We first bound the covering number of $[0,T]$, as required by Talagrand’s lower-bound principle.

\begin{lemma}\label{lem:covering}
Assume that Condition \eqref{eq constant cond} holds. Then there exists a
constant $c_{3,1}>0$ such that, for every $T>0$ and
$0<\varepsilon\le T^H$,
\begin{equation}\label{eq:covering_bound}
N\bigl([0,T],d_X,\varepsilon\bigr)
\le
c_{3,1}
T^{\frac{2H}{2\gamma+3}}
\varepsilon^{-\frac{2}{2\gamma+3}},
\end{equation}
where \(d_X\) is the canonical metric defined by
\[
d_X(s,t):=\bigl(\mathbb{E}|X(s)-X(t)|^2\bigr)^{1/2},
\]
and \(N([0,T],d_X,\varepsilon)\) denotes the smallest number of \(d_X\)-balls of radius \(\varepsilon\) needed to cover \([0,T]\).
\end{lemma}

 \begin{proof}
Set $p:=(2\gamma+3)/(2H)$ and define
\begin{equation}\label{eq:grid-def}
t_n:=\varepsilon^{1/H}n^{p}, \qquad n\ge0.
\end{equation}

By the $H$-self-similarity of $X$ and $H>0$, for every $u\in[0,t_1]$,
\begin{align}\label{eq:first-grid-interval}
    d_X(0,u)^2
=\mathbb{E}|X(u)|^2
=u^{2H}\mathbb{E}|X(1)|^2
\leq t_1^{2H}\mathbb{E}|X(1)|^2
=\varepsilon^2\mathbb{E}|X(1)|^2.
\end{align}
By Lemma \ref{lem mom}  and  \eqref{eq:grid-def}, for $n\geq2$ and $u\in[t_{n-1},t_n]$,
\begin{align}\label{eq:grid-interval-moment}
d_X(t_{n-1},u)^2
&\leq c_{2,2}
\max\{u^{2\alpha+2\beta},t_{n-1}^{2\alpha+2\beta}\}
(u-t_{n-1})^{2\gamma+3} \notag\\
&\leq c_{2,2}
\max\{t_n^{2\alpha+2\beta},t_{n-1}^{2\alpha+2\beta}\}
(t_n-t_{n-1})^{2\gamma+3}\\
&\leq c_{2,2}\varepsilon^{2}
\max\left\{n^{2(\alpha+\beta)p},(n-1)^{2(\alpha+\beta)p} \right\}
\bigl[n^{p}-(n-1)^{p}\bigr]^{2\gamma+3}\notag.
\end{align}
Since $n/2\leq n-1<\xi_n<n$, there exists a constant $c_{3,2}=\max\{1, 2^{-2(\alpha+\beta)p}\}$ such that
\begin{equation}
\max\left\{n^{2(\alpha+\beta)p},(n-1)^{2(\alpha+\beta)p}\right\}
\leq c_{3,2} n^{2(\alpha+\beta)p},
\qquad n\geq2.
\label{eq:power-maximum}
\end{equation}
Moreover, by the mean value theorem, there exist  constants $c_{3,3}, c_{3,4}$ such that
\begin{align}\label{eq:power-increment}
 c_{3,3} n^{p-1} \le n^{p}-(n-1)^{p}
=p\xi_n^{p-1}\leq c_{3,4} n^{p-1},
\qquad n\geq2,
\end{align}
where $c_{3,3}=\min\{p\cdot 2^{1-p}, p \}$ and $c_{3,4}=\max\{p,p \cdot 2^{1-p}\}.$

By \eqref{eq:grid-interval-moment}, \eqref{eq:power-maximum}, \eqref{eq:power-increment},  and $H=\alpha+\beta+\gamma+3/2$, we obtain
\begin{align}\label{eq:uniform-grid-ball}
d_X(t_{n-1},u)^2
\leq c_{2,2}c_{3,2}c_{3,4}^{2\gamma+3}
\varepsilon^{2}, \qquad
u\in[t_{n-1},t_n],\quad n\geq2.
\end{align}
Let $c_{3,5}=\max\{1,\sqrt{\mathbb{E}|X(1)|^2},c_{2,2}^{1/2}c_{3,2}^{1/2}c_{3,4}^{\gamma+3/2}\}$. From \eqref{eq:first-grid-interval} and \eqref{eq:uniform-grid-ball}, we obtain
\begin{align*}
d_X(t_{n-1},u)
\leq c_{3,5},
\varepsilon \qquad
u\in[t_{n-1},t_n],\quad n\geq1.
\end{align*}

Let
\begin{equation*}
L_\varepsilon
:=\max\{n\geq1:t_n\le T\}
=\left\lfloor
T^{\frac{1}{p}}\varepsilon^{-\frac{2}{2\gamma+3}}
\right\rfloor.
\end{equation*}
Since $t_{L_\varepsilon}\leq T<t_{L_\varepsilon+1}$, the interval $[0,T]$ can be covered by the intervals $[t_{n-1},t_n]$, $1\leq n\leq L_\varepsilon$, together with the possible remaining interval $[t_{L_\varepsilon},T]$. Therefore,
\begin{align}\label{eq:auxiliary-covering}
N([0,T],d_X,c_{3,5}\varepsilon)
\le L_\varepsilon+1
\le 2L_\varepsilon\leq 2T^{\frac{1}{p}}
\varepsilon^{-\frac{2}{2\gamma+3}}.
\end{align}

Finally, for every given $0<\varepsilon\leq T^H$, apply \eqref{eq:auxiliary-covering} with $\varepsilon/c_{3,5}$ in place of $\varepsilon$. Since $c_{3,5}\geq1$, we still have $\varepsilon/c_{3,5}\leq T^H$, and therefore
\begin{align*}
N\bigl([0,T],d_X,\varepsilon\bigr)
\leq 2c_{3,5}^{\frac{2}{2\gamma+3}}
T^{\frac{1}{p}}
\varepsilon^{-\frac{2}{2\gamma+3}}=2c_{3,5}^{\frac{2}{2\gamma+3}}
T^{\frac{2H}{2\gamma+3}}
\varepsilon^{-\frac{2}{2\gamma+3}}.
\end{align*}
Thus \eqref{eq:covering_bound} holds, and the proof is complete.
\end{proof}

\begin{proof}[\textbf{Proof of Theorem~\ref{lem small0t}}] We adapt the argument of Wang and Xiao  \cite[Proposition 5.1]{WX2022a}. 
	By \eqref{eq:covering_bound}, we have
\begin{align}\label{Eq: Ncov}
N\left([0,t],d_X,\varepsilon\right)\le  c_{3,1}\, t^{\frac{2H}{2\gamma+3}}\varepsilon^{- \frac 2 {2\gamma+3}}.
	\end{align}
Applying Lemma \ref{Lem:Ta93} with \eqref{Eq: Ncov}, we obtain the lower bound in \eqref{Eq small0t}.

We next prove  the upper bound in \eqref{Eq small0t}. Let $\{t_n\}_{n\ge0}$ be defined  by \eqref{eq:grid-def}, that is
\begin{equation}
   \label{eq:grid-def1}
t_n:=\varepsilon^{\frac1H}n^{\frac{2\gamma+3}{2H}},
\qquad n\geq0.
\end{equation}
According to Proposition~\ref{lem SLND}, \eqref{eq:power-increment}, and \eqref{eq:grid-def1}, we have
\begin{equation}\label{eq:varlower1}
\begin{split}
\operatorname{Var}\bigl(X(t_n)\mid X(t_j): j\le n-1\bigr)
&\ge
c_{2,1} t_{n}^{2\alpha+2\beta} |t_n-t_{n-1}|^{2\gamma+3} \\
&=
c_{2,1}
\varepsilon^{2}
n^{\frac{(\alpha+\beta)(2\gamma+3)}{H}}
\cdot
\bigl(n^{\frac{2\gamma+3}{2H}}-(n-1)^{\frac{2\gamma+3}{2H}}\bigr)^{2\gamma+3} \\
&\ge
c_{2,1}c_{3,3}^{2\gamma+3} \varepsilon^2.
\end{split}
\end{equation}

Let
\begin{equation*}
L_\varepsilon
:=\max\{n\geq1:t_n\le t\}
=\left\lfloor
t^{\frac{2H}{2\gamma+3}}\varepsilon^{-\frac{2}{2\gamma+3}}\right\rfloor.
\end{equation*}

Since $\varepsilon<t^H$, we have
\[
A_\varepsilon:=
t^{\frac{2H}{2\gamma+3}}
\varepsilon^{-\frac{2}{2\gamma+3}}>1.
\]
Consequently,
\[
L_\varepsilon=\lfloor A_\varepsilon\rfloor
\ge \frac12 A_\varepsilon
=
\frac12
t^{\frac{2H}{2\gamma+3}}
\varepsilon^{-\frac{2}{2\gamma+3}}.
\]

  By Anderson's inequality \cite{Anderson55} and \eqref{eq:varlower1}, we have
\begin{equation*}\label{rho}
\mathbb{P}\Bigl\{ |X(t_n)|\le\varepsilon \;\Big|\; \sigma\bigl\{X(t_j):1\le j\le n-1\bigr\} \Bigr\}
\le \ \Phi(c_{2,1}^{-1/2}c_{3,3}^{-(\gamma+3/2)}) - \Phi(-c_{2,1}^{-1/2}c_{3,3}^{-(\gamma+3/2)})=:\eta,
\end{equation*}
	where  $\Phi$ is the distribution function of a standard normal random variable.  By repeated conditioning,
	\[
\begin{aligned}
\mathbb{P}\left\{ \sup_{s\in[0,t]}|X(s)|\le\varepsilon \right\}
\le
\mathbb{P}\left\{ \max_{1\le n\le L_\varepsilon} |X(t_n)|\le\varepsilon \right\} 
\le
\eta^{L_\varepsilon}
\le
\exp\!\left(-c_{3,6} t^{\frac{2H}{2\gamma+3}}\, \varepsilon^{-\frac{2}{2\gamma+3}}\right),
\end{aligned}
\]
	where $c_{3,6} = - \frac12 \log\eta \in (0, \infty).$ The proof is complete.
\end{proof}

\subsection{Proof of Theorem \ref{lem smallt-rt+r}}   
\begin{proof}[Proof of Theorem \ref{lem smallt-rt+r}]
For $t>0$ and $r\in (0,t/2)$,  set
	 $$I(t, r):= [t - r, t + r].$$
     By Lemma \ref{lem mom}, there exist constants $c_{3,7}$ and $c_{3,8} \in (0, \infty)$ such that,   for all  $s_1,  s_2 \in I(t, r)$,
	\begin{align}\label{eq monment t}
	c_{3,7}t^{2\alpha+2\beta} |s_1 - s_2|^{2\gamma+3}\le	\mathbb E\left[(X(s_1) - X(s_2))^2\right] \le c_{3,8}t^{2\alpha+2\beta} |s_1  - s_2|^{2\gamma+3}.
	\end{align}
	It follows that  there exists a constant \(c_{3,9} \in (0, \infty)\) such  that,  for all \(0 < \varepsilon < t^{\alpha+\beta} r^{\gamma+\frac32}\),
	\[
	N(I(t, r),d_X,\varepsilon)  \leq c_{3,9} t^{\frac{2\alpha+2\beta}{2\gamma+3}} \varepsilon^{-\frac{2}{2\gamma+3}} r.
	\] 
Lemma \ref{Lem:Ta93} gives the lower bound in \eqref{Eq smallt-rt+r}. Moreover,   \eqref{eq monment t} and Proposition \ref{lem SLND} verify the hypotheses of Monrad and Rootz\'en  \cite[Theorem~2.1]{MR1995}. Therefore, that theorem yields   the upper bound in  \eqref{Eq smallt-rt+r}.

The proof is complete. 
\end{proof}

\begin{remark}\label{rem:local-scaling}
By Lemma~\ref{eq moment}, for every fixed $t>0$, there exist constants
$c_{3,10},c_{3,11}\in(0,\infty)$ such that, for all sufficiently small
$r>0$,
$$
 c_{3,10}
t^{\alpha+\beta}r^{\gamma+\frac32}\le \bigl(\mathbb{E}|X(t+r)-X(t)|^2\bigr)^{1/2}
\le c_{3, 11}
t^{\alpha+\beta}r^{\gamma+\frac32}.
$$
Thus, the restriction
$$
0<\varepsilon<
t^{\alpha+\beta}r^{\gamma+\frac32}
$$
in Theorem~\ref{lem smallt-rt+r}  describes the natural local small ball
regime.
\end{remark}

\section{Chung's LIL for $X$ at  $t>0$ }
 In this section, we prove  Chung's LIL for $X$ at an arbitrary fixed point $t>0$, as stated in Theorem~\ref{thm CLIL}, by combining the small ball probability estimates with the Lamperti transformation.

\subsection{Lamperti transformation} \label{sec:Lamperti}
Consider the Lamperti transformation of $X$, defined by
\begin{equation*}\label{Eq:U}
	U(t):= e^{-tH} X(e^t) \quad \text{for all } t \in \mathbb{R}.
\end{equation*}
By the $H$-self-similarity of $\{X(t):\, t\ge0\}$,  $U =\{U(t)\}_{t\in\mathbb R}$ is a  centered stationary Gaussian process.  Following  \cite{TX2007}, we analyze the covariance function and the spectral measure of
$U$, which will be used in the proof of Theorem~\ref{thm CLIL}.

Let $r(t):=\mathbb E[U(0)U(t)]$  denote the covariance function of $U$. A  direct  calculation gives  
\begin{align}\label{eqrt}
	r(t)=&\, e^{-tH}  \int_0^{1 \wedge e^t} s^{2\alpha} \left(\int_s^1 u^{\beta} (u-s)^{\gamma}du \right)  \left(\int_s^{e^t} v^{\beta} (v-s)^{\gamma} dv \right) ds.
\end{align}
Since $U$ is stationary, its covariance function $r$ is even. 
 
\begin{lemma}\label{key1} Suppose that Condition~\eqref{eq constant} holds. Then the covariance function $r$ has the following asymptotic behavior as $t\to\infty$.
\begin{itemize}
    \item[(i)] If $\beta+\gamma<-1$, then
    \[
    \lim_{t\to+\infty} e^{tH}r(t)=c_{4,1}.
    \]
    \item[(ii)] If $\beta+\gamma=-1$, then
    \[
\lim_{t\to+\infty}\frac{1}{t}\,e^{tH}r(t)=c_{4,2}.
    \]
    \item[(iii)] If $\beta+\gamma>-1$, then
    \[
    \lim_{t\to+\infty} e^{t(\alpha+\frac12)}r(t)=c_{4,3}.
    \]
\end{itemize}
Here $c_{4,1}$, $c_{4,2}$, and $c_{4,3}$ are strictly positive constants given explicitly in \eqref{rtcase1}, \eqref{rtcase2}, and \eqref{rtcase3}, respectively. 
\end{lemma}

\begin{proof}   
For $t>0$, applying the change of variables $v=e^{t}y$ in \eqref{eqrt}, we obtain
	\begin{equation}\label{eqrt1}
		e^{tH} r(t) =  \int_0^1 s^{2\alpha} \left( \int_s^1 u^{\beta}(u-s)^{\gamma}du\right) \left( e^{t(\beta+\gamma+1)} \int_{s e^{-t}}^1 y^{\beta}(y - se^{-t})^{\gamma}dy \right) ds.
	\end{equation} 

    We shall use the following estimate, whose proof is postponed. There exists a constant \(c_{4,4}>0\) such that for all \(0<a<b\),
\begin{equation}
\int_{a}^{b} y^\beta (y-a)^\gamma\,dy
\leq c_{4,4}
\begin{cases}
a^{\beta+\gamma+1}, & \beta+\gamma<-1,\\[2mm]
1+\log(b/a), & \beta+\gamma=-1,\\[2mm]
b^{\beta+\gamma+1}, & \beta+\gamma>-1.
\end{cases}
\label{eq:J-estimate}
\end{equation}
Applying \eqref{eq:J-estimate} together with Lemma~\ref{lem 4 MSa}, we now distinguish three cases.

\noindent\text{Case 1:}    $\beta+\gamma<-1$.  Lemma~\ref{lem 4 MSa} implies that, for every fixed $s\in(0,1]$,
	\[
	\lim_{t\to\infty} e^{t(\beta+\gamma+1)} \int_{s e^{-t}}^1 y^{\beta}(y - se^{-t})^{\gamma}dy
	= c_{4,5} s^{\beta+\gamma+1}, 
	\]
	where $c_{4,5}= B(\gamma+1, -\beta-\gamma-1)$.
    
  Moreover, by \eqref{eq:J-estimate}, we have
    \[
e^{t(\beta+\gamma+1)}\int_{s e^{-t}}^1 y^{\beta}(y - se^{-t})^{\gamma}dy \le c_{4,4} s^{\beta+\gamma+1}
\quad\text{and}\quad
\int_s^1 u^{\beta}(u-s)^{\gamma}du \le c_{4,4} s^{\beta+\gamma+1}.
\]
	Consequently,  the integrand in \eqref{eqrt1} is bounded by
$c_{4,4}^2\,s^{2\alpha+2\beta+2\gamma+2}.$
Since
$2\alpha+2\beta+2\gamma+2>-1
$, this dominating function is integrable over $(0,1)$. Hence, by the dominated convergence theorem and Fubini's theorem, we obtain
	\begin{equation}\label{rtcase1}
		\begin{split}
			\lim_{t\to+\infty} e^{tH} r(t)
			&=\int_0^1 s^{2\alpha} \left( \int_s^1 u^{\beta}(u-s)^{\gamma}du\right) \left( 	\lim_{t\to+\infty} e^{t(\beta+\gamma+1)}\int_{s e^{-t}}^1 y^{\beta}(y - se^{-t})^{\gamma}dy \right) ds\\
			&= c_{4,5} \int_0^1 s^{2\alpha+\beta+\gamma+1} \left( \int_s^1 u^{\beta}(u-s)^{\gamma}du\right) ds\\
			&= c_{4,5} \int_0^1 u^\beta \left( \int_0^u s^{2\alpha+\beta+\gamma+1} (u-s)^\gamma ds \right) du \\
			&=c_{4,5} \int_0^1 u^{2\alpha+2\beta+2\gamma+2} \left( \int_0^1 y^{2\alpha+\beta+\gamma+1}(1-y)^\gamma dy \right) du \\
			&= c_{4,5}\frac{1}{2\alpha+2\beta+2\gamma+3} B(2\alpha+\beta+\gamma+2, \gamma+1).
		\end{split}
	\end{equation} 
    \noindent\text{Case 2: $\beta+\gamma=-1$.} 
    For every fixed $s\in(0,1]$, Lemma 6.1 gives
$$
\lim_{t\to\infty}\frac1t\int_{s e^{-t}}^1 y^{\beta}(y - se^{-t})^{\gamma}dy
=1.
$$
For $t\geq1$, by \eqref{eq:J-estimate}, we have
$$
\frac1t\int_{s e^{-t}}^1 y^{\beta}(y - se^{-t})^{\gamma}dy
\leq
c_{4,4}\left(2+\log s^{-1}\right)  \quad\text{and}\quad \int_s^1 u^{\beta}(u-s)^{\gamma}du \le c_{4,4} \left(1+\log s^{-1}\right).
$$
Hence, the integrand in \eqref{eqrt1} is bounded by
$
c_{4,4}^2s^{2\alpha}\left(1+\log s^{-1}\right)\left(2+\log s^{-1}\right),
$
which is integrable over $(0,1)$ because $\alpha>-1/2$. Therefore, by the dominated convergence theorem and Fubini's theorem, we obtain
\begin{equation}
\begin{aligned}
\lim_{t\to\infty}\frac1t\,e^{tH}r(t)
&= \int_0^1 s^{2\alpha} \left( \int_s^1 u^{\beta}(u-s)^{\gamma}du\right) \left( 	\lim_{t\to+\infty} \frac1t \int_{s e^{-t}}^1 y^{\beta}(y - se^{-t})^{\gamma}dy \right) ds\\
&= \int_0^1 s^{2\alpha}
\Bigl(\int_s^1 u^\beta(u-s)^\gamma du\Bigr)ds\\
&= \int_0^1 u^\beta \left( \int_0^u s^{2\alpha} (u-s)^\gamma ds \right) du \\
&= \int_0^1 u^{2\alpha+\beta+\gamma+1} du \cdot B(2\alpha+1,\gamma+1) \\
&= \frac{B(2\alpha+1,\gamma+1)}{2\alpha+\beta+\gamma+2} = \frac{B(2\alpha+1,\gamma+1)}{2\alpha+1}.
\end{aligned}
\label{rtcase2}
\end{equation}
\text{Case 3: $\beta+\gamma>-1$.}
Lemma~\ref{lem 4 MSa} implies that, for every fixed $s\in(0,1]$,
	\[
	\lim_{t\to\infty}  \int_{s e^{-t}}^1 y^{\beta}(y - se^{-t})^{\gamma}dy
	= \frac{1}{\beta+\gamma+1}. 
	\] 
  Moreover, by \eqref{eq:J-estimate}, we have
\[
\int_{s e^{-t}}^1 y^{\beta}(y - se^{-t})^{\gamma}dy \le c_{4,4}
\quad\text{and}\quad
\int_s^1 u^{\beta}(u-s)^{\gamma}du \le c_{4,4}.
\]
	Consequently,  the integrand in \eqref{eqrt1} is bounded by
$c_{4,4}^2\,s^{2\alpha},$
which is integrable over $(0,1)$.
Thus,
\begin{equation}
\begin{aligned}
\lim_{t\to\infty}e^{t(\alpha+\frac12)}r(t)
&=\lim_{t\to\infty} e^{-t(\beta+\gamma+1)}\,e^{tH}r(t)\\
			&= \int_0^1 s^{2\alpha} \left( \int_s^1 u^{\beta}(u-s)^{\gamma}du\right) \left( 	\lim_{t\to+\infty} \int_{s e^{-t}}^1 y^{\beta}(y- se^{-t})^{\gamma}dy \right) ds\\
&=
\frac{1}{\beta+\gamma+1}
\int_0^1 s^{2\alpha}
\left(\int_s^1 u^\beta (u-s)^\gamma\,du\right)ds\\
&=
\frac{B(2\alpha+1,\gamma+1)}
{(\beta+\gamma+1)(2\alpha+\beta+\gamma+2)}.
\end{aligned}
\label{rtcase3}
\end{equation}
Combining \eqref{rtcase1}, \eqref{rtcase2}, and \eqref{rtcase3} yields Lemma~\ref{key1}.
 We next prove \eqref{eq:J-estimate}.   
 
 For $\beta+\gamma<-1$, using the change of variables $y=a/(1-z)$, we obtain
\begin{align*}\label{eq:kernel_beta_integral}
\int_a^b y^\beta (y-a)^\gamma dy
= a^{\beta+\gamma+1}
\int_0^{1-a/b} z^\gamma (1-z)^{-\beta-\gamma-2}dz
\le B(\gamma+1,-\beta-\gamma-1) \, a^{\beta+\gamma+1}.
\end{align*}

For $\beta+\gamma=-1$, we have $\beta<0$. The change of variables $y=a(1+z)$ yields
\[
\int_a^b y^\beta(y-a)^\gamma\,dy
=
\int_0^{(b-a)/a}(1+z)^\beta z^{-1-\beta}\,dz.
\]
If $(b-a)/a\le1$, then
\[
\int_0^{\frac{b-a}{a}} (1+z)^\beta z^{-1-\beta}\,dz
\le
\int_0^1 z^{-1-\beta}\,dz
=-\frac1\beta.
\]
If $(b-a)/a>1$, then, since
$(1+z)^\beta\le z^\beta$ for $z\ge1$,
\[
\begin{aligned}
\int_0^{\frac{b-a}{a}} (1+z)^\beta z^{-1-\beta}\,dz
&\le
\int_0^1 z^{-1-\beta}\,dz
+\int_1^{\frac{b-a}{a}}z^{-1}\,dz \\
&\le
-\frac1\beta+\log\frac ba.
\end{aligned}
\]
Thus, there exists $c_{4,6}>0$ such that, for all $0<a\le b$,
\[
\int_a^b y^\beta (y-a)^\gamma\,dy
\le c_{4,6}\left(1+\log\frac{b}{a}\right).
\]


For $\beta+\gamma>-1$, we have 
\[
\begin{aligned}
\int_a^b y^\beta (y-a)^\gamma dy
= \int_0^{b-a} (a+z)^\beta z^\gamma dz 
&\le
\begin{cases}
\displaystyle b^\beta \int_0^{b-a} z^\gamma dz, & \beta>0,\\[6pt]
\displaystyle \int_0^{b-a} z^{\beta+\gamma} dz, & \beta\le 0,
\end{cases} \\
&\le \max\left\{ (\beta+\gamma+1)^{-1}, (\gamma+1)^{-1} \right\} b^{\beta+\gamma+1}.
\end{aligned}
\]
    The proof is complete.
\end{proof}

By Bochner's theorem (see \cite[p.~220]{Lo}), there exists a finite measure $F$ on $\mathbb R$ such that
\begin{equation}\label{eq:r-spectral}
r(t)=\int_{\mathbb R}e^{it\omega}\,F(d\omega),
\qquad t\in\mathbb R.
\end{equation}
The measure $F$ is the spectral measure of $U$. Since $r\in L^1(\mathbb R)$, $F$ is absolutely continuous with respect to Lebesgue measure and has a continuous density $f$. Moreover, since $r$ is even, the Fourier inversion formula gives
$$
f(\omega)
=
\frac{1}{2\pi}\int_{\mathbb R}e^{-it\omega}r(t)\,dt
=
\frac{1}{\pi}\int_0^\infty r(t)\cos(t\omega)\,dt,
\qquad \omega\in\mathbb R.
$$

It is well known that the stationary Gaussian process $U$ admits  the stochastic integral representation:
\begin{align}\label{eq U int}
	U(t)=\int_{\mathbb R} e^{i\omega t}\,\widehat W(d\omega)    \ \ \ \ \text{for all } t\in \mathbb R,
\end{align}
where  $\widehat W$  is a complex Gaussian measure with control measure $F$. Consequently, for all $s, t \in \mathbb R$,
\begin{equation}\label{Eq:Uvar}
	\begin{split}
		\mathbb E \big[ \big(U (s) -U (t)\big)^2\big] &= 2 \big( r(0) - r(t-s) \big)\\
		&=2 \int_{\mathbb R} \big[1 - \cos \big((s-t) \omega\big)\big]\, f(\omega) d \omega.
	\end{split}
\end{equation}

The following lemma provides two-sided bounds for  $\mathbb E \left[ \left(U (s) -U (t)\right)^2\right].$

\begin{lemma}\label{lem r}  
	Assume that  Condition \eqref{eq constant cond} holds. Then,  for every  $T> 0$, there exist positive constants $\varepsilon_0$, $c_{4,7}$, and $c_{4,8}$ such that,  for all $ s, t \in [0, T]$ satisfying  $|s-t| \le \varepsilon_0$,
	\begin{equation}\label{eq r}
		\begin{split}
			c_{4,7}|s-t|^{2 \gamma+3}  \le\, \mathbb E \left[ \left(U (s) -U (t)\right)^2\right]  \le  \,
			c_{4,8}|s-t|^{2\gamma+3}.
		\end{split}
	\end{equation} 
\end{lemma}
\begin{proof}  Following the argument of \cite[Lemma~4.1]{WX2022a} and using the stationarity of $U$, it suffices to estimate
$$\mathbb E \left[ \left(U (t) -U (0)\right)^2\right], \qquad t>0.$$
 
For sufficiently small $t>0$, it follows from Lemma \ref{lem mom} and the inequality
$(x+y)^2\leq 2x^2+2y^2$ that there exist positive constants $c_{4,9}$, $c_{4,10}$, and $c_{4,11}$ such that 
\begin{equation*}\label{Eq:Umin}
\begin{aligned}
\mathbb{E}\left[\left(U(t)-U(0)\right)^2\right]
&=
\mathbb{E}\left[
\left(
X(e^t)-X(1)
+\bigl(e^{-Ht}-1\bigr)X(e^t)
\right)^2
\right] \\
&\leq
2\mathbb{E}\left[\left(X(e^t)-X(1)\right)^2\right]
+
2\bigl(e^{-Ht}-1\bigr)^2
\mathbb{E}\left[X(e^t)^2\right] \\
&\leq
2c_{2,2}
\max\left\{e^{(2\alpha+2\beta)t},1\right\}
\left(e^t-1\right)^{2\gamma+3}
+
2\left(e^{Ht}-1\right)^2
\mathbb{E}\left[X(1)^2\right] \\
&\leq
c_{4,9}t^{2\gamma+3}
+
c_{4,10}t^2 \\
&\leq
c_{4,11}t^{2\gamma+3}.
\end{aligned}
\end{equation*}
Here, we used Lemma \ref{lem mom}, 
$H$-self-similarity,  $
e^t-1\sim t$,
and  $e^{Ht}-1\sim Ht$
 as $t\downarrow0$,
and   $2\gamma+3\in(1,2)$.

Similarly, using  $(x+y)^2\ge \frac12 x^2-y^2$,   we obtain a constant $c_{4,12}>0$ such that for all   sufficiently small $t>0$, 
\begin{align*}
\mathbb{E}\left[\left(U(t)-U(0)\right)^2\right]
\ge \frac12
\mathbb{E}\left[\left(X(e^t)-X(1)\right)^2\right]
-\left(e^{-Ht}-1\right)^2
\mathbb{E}\left[X(e^t)^2\right]
\ge c_{4,12}t^{2\gamma+3}.
\end{align*} 

The proof is complete.
\end{proof}

\begin{lemma}\label{lem spectral}  
	Assume that Condition \eqref{eq constant cond} holds. Then there exist  constants $u_0>0$ and 	$c_{4,13}, c_{4,14}>0$ such that, for every $u > u_0$,
	\begin{align}\label{Eq: sp1}
		\int_{|\omega|<u} \omega^2 f(\omega)d\omega &\le c_{4,13} u^{-1-2\gamma}, \\
		\label{Eq: sp2}
		\int_{|\omega|\ge u} f(\omega)d\omega &\le c_{4,14} u^{-(2\gamma+3)}.
	\end{align}
\end{lemma} 

\begin{proof} 
	By the truncation inequalities of Lo\`eve \cite [p. 209]{Lo}, there exists a constant $K>0$ such that, for all $u>0$,
	\begin{equation*}
	\begin{aligned}
		\int_{|\omega| < u} \omega^2 f(\omega)d\omega
		&\le K u^2 \int_{\mathbb R} (1 - \cos(\omega/u)) f(\omega)d\omega, \\
		\int_{|\omega| \ge u} f(\omega)d\omega
		&\le K u \int_0^{1/u} dv \int_{\mathbb R} (1 - \cos(v\omega)) f(\omega)d\omega.
	\end{aligned}
	\end{equation*}
  Choose $u_0>0$ sufficiently large  so that $$1/u_0 \le \varepsilon_0,$$
  where $\varepsilon_0$ is the constant appearing in Lemma~\ref{lem r}.
  Then, for every $u\ge u_0$,  \eqref{Eq:Uvar} and \eqref{eq r} yield 
	\begin{align*}
		\int_{|\omega| < u} \omega^2 f(\omega)d\omega
		\le \frac{K}{2} u^2 \mathbb E[(U(1/u)-U(0))^2] 
		\le \frac{K c_{4,8}}{2} u^{-1-2\gamma},
	\end{align*} 
	and
	\begin{align*}
		\int_{|\omega| \ge u} f(\omega)d\omega
		&\le \frac{K}{2} u \int_0^{1/u} \mathbb E[(U(v)-U(0))^2] dv \\
		&\le \frac{K c_{4,8}}{2} u \int_0^{1/u} v^{2\gamma+3}dv= \frac{K c_{4,8}}{2(2\gamma+4)} u^{-(2\gamma+3)}.
	\end{align*}
	The proof is complete. 
\end{proof}

\subsection{Proof of Theorem~\ref{thm CLIL} }

   We combine the arguments of Tudor and Xiao  \cite[Theorem~3.1]{TX2007} and Wang and Xiao \cite[Proposition~6.1]{WX2022a} with the small ball estimates  and the Lamperti transformation. 
  
The following zero-one law shows that the left-hand side of \eqref{eq CLIL1} 
is almost surely a constant, which  may be $0$ or $\infty$.

\begin{lemma}\label{lem:01lawZ} 
	Assume that  Condition \eqref{eq constant cond} holds.  Then there exists a constant $\kappa_5'\in[0,\infty]$  such that,  for every  fixed $t > 0$,
	\begin{equation}\label{Eq:Chung01law} 
		\liminf_{r\rightarrow0+}\sup_{ |s|\le r} \frac{ (\log \log r^{-1})^{\gamma+\frac32}} {r^{\gamma+\frac32}}|X(t+s)-X(t)|
		=\kappa_5' t^{\alpha+\beta}, \ \  \ \ \ \text{a.s.}
	\end{equation}
\end{lemma}

\begin{proof} We begin with the stationary Gaussian process $U = \{U(s)\}_{s \in \mathbb R}$ defined by the spectral representation  \eqref{eq U int}. 
Following  the proof of Lemma 2.1 in \cite{WSX},   for every $n \ge 1$, define
	\[
	U_n(s) = \int_{n-1 \le |\omega| < n} e^{i\omega s}\widehat W(d\omega).
	\]
	Since the sets $\{n-1\le |\omega|<n\}$ are pairwise disjoint and $\widehat {W}$ is  independently scattered, the Gaussian processes $$U_n = \{U_n(s)\}_{s \in \mathbb R}, \quad n \ge 1$$ are mutually independent. Moreover,  
	\[
	U(s) = \sum_{n=1}^\infty U_n(s), 
	\]
	where the series  converges  uniformly almost surely    on every compact interval of $\mathbb R$. As shown in \cite[Lemma 2.1]{WSX}, the sample paths of each $U_n$ are almost surely Lipschitz continuous on every compact interval.  Consequently, for every fixed $N\in\mathbb{N}$,
$$
\sup_{|s|\leq r}
\left|
\sum_{n=1}^{N}\bigl(U_n(x+s)-U_n(x)\bigr)
\right|
=O(r)
\qquad\text{as }r\downarrow0,
\quad\text{a.s.}
$$
       Since $\gamma+3/2 \in (1/2,1)$,  for every $x\in\mathbb R$,
	\[
	\liminf_{r\to0+}\sup_{|s|\le r} 
	\frac{(\log\log r^{-1})^{\gamma+\frac32}}{r^{\gamma+\frac32}}\left|\sum_{n=1}^N (U_n(x+s)-U_n(x))\right|=0.
	\]
Consequently,  for every $x \in\mathbb R $ and $c\ge 0$, the event 
	$$E_c(x) =\bigg\{ \liminf_{r\rightarrow0+}\sup_{|s|\le r}  \frac{ (\log \log r^{-1})^{\gamma+\frac32} } {r^{\gamma+\frac32}} |U(x+s)-U(x)|\le c\bigg\}$$  
    is  unchanged by the addition or removal of any finite number of the processes $U_n$. Hence, 
 $E_c(x)$ belongs to   the tail $\sigma$-field generated by the independent sequence $\{U_n:\, n \ge 1\}$.  Kolmogorov's zero-one law therefore implies  that 
 \[
\mathbb{P}\{ E_c(x) \} \in \{0, 1\}.
\]

Define
\[
\kappa_5'
:=\inf\left\{
c\geq0:\mathbb{P}\{ E_c(x) \}=1
\right\},
\]
with the convention that $\inf\varnothing=\infty$. Since the events $E_c(x)$ are increasing in $c$, it follows that $\kappa_5'\in[0,\infty]$ and
\begin{equation}\label{01U}
\liminf_{r\to0+}\sup_{|s|\leq r}
\frac{
\displaystyle \bigl(\log\log r^{-1}\bigr)^{\gamma+\frac32} 
}{
\displaystyle r^{\gamma+\frac32}}|U(x+s)-U(x)|
= \kappa_5',
\qquad \text{a.s.}
\end{equation}
By  stationarity, the distribution  of the liminf in \eqref{01U} does not depend on  $x$; hence neither does  $\kappa_5'$. 

	 We now transfer   \eqref{01U} through the Lamperti transformation.  Fix $t>0$. For  $0 < r < t$  and $|s|\le r$, we write 
	\begin{equation}\label{01U2}
		\begin{split}
			&X(t+s) - X(t) \\
			= &\, \left[(t+s)^H - t^H\right] U(\log (t+s)) + t^H \left[U(\log (t+s)) - U(\log t)\right]\\
			 = &\, \left[(t+s)^H - t^H\right] U(\log (t+s)) + t^H \left[U\left(\log t + \log ( 1 +  s/t) \right) - U(\log t)\right].
		\end{split}
	\end{equation} 
  Since $U$ has continuous sample paths and is therefore locally bounded almost surely, the first term in \eqref{01U2} is $O(r)$ almost surely and is negligible under the normalization in \eqref{Eq:Chung01law}.
       
  For the second term, set $x:=\log t$ and define
$$
a_r:=\log\left(1+\frac rt\right),
\qquad
b_r:=-\log\left(1-\frac rt\right).
$$
Since $a_r\leq b_r$, we have
$$
[-a_r,a_r]
\subseteq
[-b_r,a_r]
\subseteq
[-b_r,b_r].
$$
Consequently,
\begin{equation}\label{eq:interval-sandwich}
\begin{split}
\sup_{|u|\leq a_r}|U(x+u)-U(x)|
\leq &\, 
\sup_{|s|\leq r}
\left|U\left(x+\log\left(1+\frac{s}{t}\right)\right)-U(x)\right|\\
\leq &\, 
\sup_{|u|\leq b_r}|U(x+u)-U(x)|.
 \end{split}
\end{equation}
Moreover,
$$
\frac{a_r}{r}\longrightarrow\frac1t,
\qquad
\frac{b_r}{r}\longrightarrow\frac1t,
$$
and
$$
\frac{\log\log a_r^{-1}}{\log\log r^{-1}}
\longrightarrow1,
\qquad
\frac{\log\log b_r^{-1}}{\log\log r^{-1}}
\longrightarrow1.
$$
Applying \eqref{01U} to the lower and upper bounds in \eqref{eq:interval-sandwich}, and using the preceding asymptotic relations, we obtain
$$
\begin{aligned}
&\liminf_{r\to0^+}
\frac{(\log\log r^{-1})^{\gamma+\frac32}}{r^{\gamma+\frac32}}
\sup_{|s|\leq r}
\left|U\left(x+\log\left(1+\frac{s}{t}\right)\right)-U(x)\right|=\kappa_5' t^{-(\gamma+\frac32)}
\qquad\text{a.s.}
\end{aligned}
$$
Multiplying by $t^H$ and using
$
H-\gamma-\frac32=\alpha+\beta,
$
we conclude from \eqref{01U2} that
$$
\liminf_{r\to0+}\sup_{|s|\leq r}
\frac{
\displaystyle
\bigl(\log\log r^{-1}\bigr)^{\gamma+\frac32} 
}{
\displaystyle
r^{\gamma+\frac32}}|X(t+s)-X(t)|
= \kappa_5't^{\alpha+\beta},
\qquad\text{a.s.}
$$
This proves \eqref{Eq:Chung01law} and completes the proof.
\end{proof}

It follows from Lemma \ref{lem:01lawZ} that Theorem \ref{thm CLIL} will be established if we show $\kappa_5' \in (0, \infty)$.

\begin{proof}[\bf{Proof of Theorem \ref{thm CLIL}}] 
Since the constant \(\kappa_5'\) in Lemma \ref{lem:01lawZ} is independent of \(t\), it suffices to assume \(t>1\).  We  first establish  the lower bound.
	For each integer $n \ge 1$, set $r_n = e^{-n}$.  Fix a constant  $\delta\in (0, \kappa_4)$ and  define the event
	\[
	A_n = \bigg\{ \sup_{|s|\le r_n} \big|X(t+s)-X(t)\big| \le
	\delta^{\gamma+\frac32}\, t^{\alpha+\beta}\, r_n^{\gamma+\frac32}(\log \log (r^{-1}_n))^{-\gamma+\frac32}\bigg\}.
	\]
    By 	Theorem \ref{lem smallt-rt+r},  for all sufficiently large $n$,
	\begin{equation*}\label{Eq:An1}
		\mathbb P\{A_n\} \le \exp\bigg(- \frac{\kappa_4 } {\delta} \log n \bigg)
		= n^{- \frac{\kappa_4 }{\delta}}.
	\end{equation*}
	Since $\sum_{n=1}^{\infty} \mathbb P\{A_n\} < \infty$, the Borel-Cantelli lemma
	implies
	\begin{equation}\label{Eq:LIL-lb1}
		\liminf_{n\to \infty} \sup_{|s|\le r_n}\frac{(\log \log r^{-1}_n)^{\gamma+\frac32} } {r_n^{\gamma+\frac32}}\big|X(t+s)-X(t)\big|
		\ge  \delta^{\gamma+\frac32}\, t^{\alpha+\beta} \qquad \hbox{ a.s.}
	\end{equation}
	It follows from \eqref{Eq:LIL-lb1} and a standard monotonicity argument that
	\begin{equation*}\label{Eq:LIL-lb2}
		\liminf_{r \to 0+}\sup_{|s| \le   r} \frac{ (\log \log r^{-1})^{\gamma+\frac32} }
		{r^{\gamma+\frac32}} \big|X(t+s)-X(t)\big| \ge c_{4,15}\, t^{\alpha+\beta} \qquad \hbox{ a.s., }
	\end{equation*}
	for some   constant $c_{4,15}>0$,  which is independent of $t>0$.
	
The proof of the  upper bound is more delicate because of the dependence structure of $X$. 	Following the approach of Tudor and Xiao \cite{TX2007}, we use a stochastic integral representation of $X$ to construct a sequence of independent Gaussian processes. More precisely, 
	\begin{equation*}\label{Eq:Rep1}
		X(t) = t^{H} \int_{\mathbb R} e^{i \omega \log t}\, 
		\widehat W(d\omega),\ \ \ \ \ t>0.
	\end{equation*}
This follows from \eqref{eq U int}  and the Lamperti transformation.

	For every integer $n \ge 2$, set
	\begin{equation*}\label{Eq:tn}
		t_n := n^{-n} \ \ \ \hbox{ and }\ \ \ d_n := n^{n-1/2-\gamma}.
	\end{equation*}
    To establish the upper bound, it suffices to show that there exists a finite constant $c_{4,16}>0$ such that 
	\begin{equation}\label{Eq:UP}
		\liminf_{n\to \infty}\sup_{|s|\le 
			t_n} \frac{(\log \log t_n^{-1})^{\gamma+\frac32}} {t_n^{\gamma+\frac32}}\big|X(t+s)-X(t)\big|
		\le c_{4,16}t^{\alpha+\beta} \qquad \hbox{ a.s.}
	\end{equation}

    For each integer $n\ge2$,   define the Gaussian processes $X_n$ and $\widetilde{X}_n$ by
	\begin{equation*}\label{Eq:Xn1}
		X_n(u) = u^{H} \int_{|\omega| \in (d_{n-1}, d_n]} e^{i \omega
			\log u}\, \widehat W(d\omega),\ \quad u>0,
	\end{equation*}
	and
\begin{equation*}\label{Eq:Xn2}
		\widetilde{X}_n(u) = u^{H} \int_{|\omega| \notin (d_{n-1}, d_n]}
		e^{i \omega\log u}\, \widehat W(d\omega), \ \quad u>0,
	\end{equation*}
	respectively. By construction,
    $$X(u) = X_n(u) + \widetilde {X}_n(u),  \ \ \quad u>0.$$  Since the frequency bands  
$$
\{\omega\in\mathbb{R}:d_{n-1}<|\omega|\leq d_n\},
\qquad n\geq2,
$$
are pairwise disjoint and $\widehat{W}$ is an independently scattered Gaussian random measure, the processes $\{X_n\}_{n\ge2}$ are mutually independent. Moreover, for each $n\geq2$, the processes $X_n$ and $\widetilde{X}_n$ are independent.

	Define
    $$h(r) := r^{\gamma+\frac32}\, \big(\log \log r^{-1}\big)^{-({\gamma+\frac32})}.$$
    We shall establish 	the following two claims:
	\begin{itemize}
		\item[(i)]\ There exists a constant $\delta > 0$ such that
		\begin{equation}\label{Eq:UP1}
			\sum_{n=2}^\infty \mathbb P\bigg\{\sup_{|s|\le t_n} \big|X_n(t+s)-X_n(t)\big|
			\le \delta^{\gamma+\frac32}\,t^{\alpha+\beta}\, h(t_n)\bigg\} = \infty.
		\end{equation}
		\item[(ii)]\ For every $\varepsilon> 0$,
		\begin{equation}\label{Eq:UP2}
			\sum_{n=2}^\infty \mathbb P\bigg\{\sup_{|s| \le   t_n}
			\big|\widetilde{X}_n(t+s)-\widetilde{X}_n(t)\big| > \varepsilon\,t^{\alpha+\beta}\, h(t_n)\bigg\} < \infty.
		\end{equation}
	\end{itemize}

By independence and the second Borel–Cantelli lemma, the events in \eqref{Eq:UP1} occur infinitely often, whereas \eqref{Eq:UP2}, together with the first Borel–Cantelli lemma, shows that the remainder is negligible.
 Hence, \eqref{Eq:UP} follows.

	It remains to verify Claims (i) and (ii). By Anderson's inequality \cite{Anderson55} and Theorem \ref{lem smallt-rt+r}, for all sufficiently large $n$,
	\begin{equation*}\label{Eq:UP3}
		\begin{split}
			&\mathbb P\bigg\{\sup_{|s| \le t_n} \big|X_n(t+s)-X_n(t)\big| \le \delta^{\gamma+\frac32}\, t^{\alpha+\beta}\,
			h(t_n)\bigg\}\\
			 \ge & \,\mathbb P\bigg\{\sup_{|s| \le t_n} \big|X(t+s)-X(t)\big| \le \delta^{\gamma+\frac32}\, t^{\alpha+\beta}\, h(t_n)\bigg\}\\
			\ge & \, \exp\Big(- \frac{\kappa_3} {\delta} \log (n \log n) \Big)\\
		=\, & \big(n \log n\big)^{-\frac{ \kappa_3}{\delta}}.
		\end{split}
	\end{equation*}
	Hence (i) holds for $\delta \ge  \kappa_3$.

To prove Claim (ii),   we decompose the second moment of the increment of  $\widetilde {X}_n$ into its low- and high-frequency components. 
For  all \(u,v\in[t-t_n,t+t_n]\), we have
 
\begin{equation}\label{eq:Jc}
\begin{aligned}
\mathbb{E}\left|\widetilde X_n(u)-\widetilde X_n(v)\right|^2
={}&
\int_{|\omega|\leq d_{n-1}}
\left|u^H e^{i\omega\log u}-v^H e^{i\omega\log v}\right|^2 f(\omega)\,d\omega\\
&+
\int_{|\omega|>d_n}
\left|u^H e^{i\omega\log u}-v^H e^{i\omega\log v}\right|^2 f(\omega)\,d\omega\\
=: &\, \mathcal J_1 + \mathcal J_2.
\end{aligned}
\end{equation}
For all sufficiently large \(n\), we have \([t-t_n,t+t_n]\subset[t/2,3t/2]\). By \eqref{Eq: sp2}, there exists a constant \(c_{4,17}>0\) such that
\begin{equation}
\begin{aligned}
\mathcal J_2
&\leq 2\left(u^{2H}+ v^{2H}\right)\int_{|\omega|>d_n}f(\omega)\,d\omega\\
&\leq c_{4,17} d_n^{-(2\gamma+3)}
= c_{4,17} n^{-(2\gamma+3)(n-1/2-\gamma)}.
\end{aligned}
\label{eq:J2}
\end{equation}

Set \(\rho:=|u-v|\). For the low-frequency term $\mathcal J_1$, we use the elementary inequalities:
\begin{align*}
\left|x^H-y^H\right|
&\le |H| \max\{x^{H-1}, y^{H-1}\}|x-y|,
&&(x,y>0),\\
1-\cos x &\le x^2/2,
&&(x\in\mathbb R),\\
|\log x-\log y| &\le \max\{x^{-1}, y^{-1}\} |x-y|,
&& (x,y>0).
\end{align*}
Since \(u,v\in[t/2,3t/2]\) and \(\rho\le 2t_n\), there exist constants \(c_{4,18},c_{4,19}\in(0,\infty)\) such that
\begin{equation}
\begin{aligned}\label{eq:J1}
\mathcal J_1
&=
\int_{|\omega|\le d_{n-1}}
\left[(u^H-v^H)^2 + 2u^Hv^H(1-\cos(\omega\log(u/v)))\right] f(\omega)\,d\omega\\
&\le
H^2\max\{u^{2(H-1)},v^{2(H-1)}\}\rho^2 \int_{\mathbb R} f(\omega)\,d\omega\\
&\quad + u^Hv^H  \max\{u^{-2}, v^{-2}\} \rho^2\int_{|\omega|\le d_{n-1}} \omega^2 f(\omega)\,d\omega\\
&\le
c_{4,18}\left(
\int_{\mathbb R} f(\omega)\,d\omega
+
\int_{|\omega|\le d_{n-1}} \omega^2 f(\omega)\,d\omega
\right)n^{-2n}\\
&\le c_{4,18}\left(\mathbb{E}(X(1)^2)+d_{n-1}^{-1-2\gamma}\right)n^{-2n}\\
&\le c_{4,18}\left(\mathbb{E}(X(1)^2)+n^{-(2\gamma+1)(n-3/2-\gamma)}\right)n^{-2n}\\
&\le c_{4,19} n^{-(2\gamma+3)(n-1/2-\gamma)}.
\end{aligned}
\end{equation}
In the fourth line, we applied \eqref{eq:r-spectral} and \eqref{Eq: sp1}. In the last two lines, we used the facts that
\[
-(2\gamma+1)\left(n-3/2-\gamma\right)>0
\quad\text{and}\quad
2n\ge (2\gamma+3)\left(n-1/2-\gamma\right),
\qquad n\ge2.
\]
From \eqref{eq:Jc}, \eqref{eq:J2}, and \eqref{eq:J1}, there exists a constant \(c_{4,20}>0\) such that
\begin{align}\label{widetildex1}
\mathbb E|\widetilde X_n(u)-\widetilde X_n(v)|^2
\le c_{4,20} n^{-(2\gamma+3)(n-1/2-\gamma)}.
\end{align}
Moreover, from Lemma~\ref{lem mom}, there exists a constant \(c_{4,21}>0\) such that
\begin{align}\label{widetildex2}
\mathbb E|\widetilde X_n(u)-\widetilde X_n(v)|^2
\le \mathbb E| X_n(u)-X_n(v)|^2\le c_{4,21}\rho^{2\gamma+3}.
\end{align}
Combining \eqref{widetildex1} and \eqref{widetildex2} yields a constant \(c_{4,22}>0\) such that, for all sufficiently large \(n\) and all \(u,v\in[t-t_n,t+t_n]\),
	
\[
\mathbb E|\widetilde X_n(u)-\widetilde X_n(v)|^2
\le 
c_{4,22}\min\left\{|u-v|^{2\gamma+3},\,
n^{-(2\gamma+3)(n-1/2-\gamma)}
\right\}=:\varphi_{n}(\rho)^2.
\]

Let \[
Z_n^+(s):=\widetilde X_n(t+s)-\widetilde X_n(t)
\quad\text{and}\quad
Z_n^-(s):=\widetilde X_n(t-s)-\widetilde X_n(t).
\] For either choice of sign and  all $0\le s\le s+\rho\le t_n$, the estimate above gives
\begin{equation}
\mathbb{E}\big|Z_n^{\pm}(s+\rho)-Z_n^{\pm}(s)\big|^2
\le \varphi_n(\rho)^2.
\label{eq:Z-increment}
\end{equation}

	Set $x_n:=(8\log n)^{1/2}$. Fix $\varepsilon>0$. For each integer $p\ge1$,   define 
	\begin{align*}
		\theta_n(p):=\frac{\varepsilon}{2} t^{\alpha+\beta} (p+1)^{-2} h(t_n)/\varphi_n\left(t_n n^{-2^p}\right)\ \ \ \ \text{for  all } p\ge1.
	\end{align*} 
For all sufficiently large \(n\), we have
\[
\theta_n(p)>4(\log n)^{1/2}2^{p/2}\qquad\forall p\ge1,
\]
and
\[
x_n\varphi_n(t_n)+\sum_{p=1}^{\infty}\theta_n(p)\varphi_n(t_n n^{-2^p})
<\varepsilon t^{\alpha+\beta}h(t_n).
\]
Moreover,
\[
\sum_{n=2}^{\infty}n^2e^{-x_n^2/2}
+\sum_{n=2}^{\infty}\sum_{p=1}^{\infty} n^{2^{p+1}}e^{-8(\log n)2^p}
<\infty.
\]
Applying Lemma \ref{lem: Fern} with \eqref{eq:Z-increment} yields
\[
\sum_{n=2}^{\infty}\mathbb P\left\{
\sup_{|s|\le t_n}|\widetilde X_n(t+s)-\widetilde X_n(t)|
>\varepsilon t^{\alpha+\beta}h(t_n)
\right\}<\infty.
\]
This proves \eqref{Eq:UP2} and completes the proof.
\end{proof}

\section{Proof of Theorem~\ref{thm CLTL0infinity}}
Theorem \ref{thm CLTL0infinity}  follows from the lower-class criteria in  Theorems \ref{thoLLC0} and \ref{thoLLCinfty}. Their proofs use Talagrand’s method \cite{talagrand1996lower} and its extensions \cite{elnouty2004,elnouty2011,LWW,wang2022lower}. We first recall the notion of lower classes for stochastic processes, following   R\'ev\'esz \cite{revesz1990random}.
\begin{definition}
Let $\{M(t)\}_{t\geq 0}$ be a stochastic process, and let $\xi$ be a real-valued function defined on $(0,\infty)$.
\begin{enumerate}
\item[(a)] The function $\xi$ is said to belong to the \textit{lower-lower class} of $M$ at $\infty$ \textup{(resp. at $0$)}, denoted by
$\xi\in\textit{LLC}_{\infty}(M)$
\textup{(resp. $\xi\in\textit{LLC}_{0}(M)$)}, if, almost surely, there exists a random variable $t_0=t_0(\omega)>0$ such that
$M(t)\geq\xi(t)$ for all $t>t_0$
\textup{(resp. for all $0<t<t_0$)}.

\item[(b)] The function $\xi$ is said to belong to the \textit{lower-upper class} of $M$ at $\infty$ \textup{(resp. at $0$)}, denoted by
$\xi\in\textit{LUC}_{\infty}(M)$
\textup{(resp. $\xi\in\textit{LUC}_{0}(M)$)}, if, almost surely, there exists a sequence $\{t_n(\omega)\}_{n\geq1}$ such that
$t_n(\omega)\uparrow\infty$
\textup{(resp. $t_n(\omega)\downarrow0$)}
as $n\to\infty$, and
$M(t_n(\omega))\leq\xi(t_n(\omega))$
for every $n\in\mathbb N$.
\end{enumerate}
\end{definition}

Throughout this section, set
$$
\lambda:=\gamma+\frac32,
\qquad
q:=\frac{2}{2\gamma+3},
$$
and define
\begin{equation*}\label{def:Msup}
	M(t) := \sup_{s \in [0,t]} |X(s)|.
\end{equation*}
By the $H$-self-similarity of $X$, the  small ball probability   function of $M$ is given by 
\begin{equation}\label{Eq:sb31}
\varphi(\varepsilon) := \mathbb{P}\left\{ M(1) \leq \varepsilon \right\} = \mathbb{P}\left\{ M(t) \leq \varepsilon t^{H} \right\},\qquad t>0.
\end{equation}
 
\begin{theorem}\label{thoLLC0}
	Assume that Condition \eqref{eq constant cond} holds. Let $\xi : (0, e^{-e}] \to (0, \infty)$ be a nondecreasing continuous function.
	\begin{enumerate}
		\item[(a)] \textit{(Sufficiency)}. If
		\begin{align}\label{LLC0Sufficiency}
			\frac{\xi(t)}{t^H} \text{ is bounded and } 
			I_0(\xi) := \int_0^{e^{-e}} \left( \frac{\xi(t)}{t^H} \right)^{-\frac{2}{2\gamma+3}} \varphi\!\left( \frac{\xi(t)}{t^H} \right) \frac{dt}{t} < +\infty,
		\end{align}
		then $\xi \in LLC_0(M)$.	
		\item[(b)] \textit{(Necessity)}. Conversely, if $\xi \in LLC_0(M)$ and there exists a constant $c_{5,1} \ge 1$ such that $\xi(2t) \le c_{5,1}\xi(t)$ for all $t \in (0, e^{-e}/2]$, then \eqref{LLC0Sufficiency} holds.
	\end{enumerate}
	
\end{theorem}

\begin{theorem}\label{thoLLCinfty}
	Assume that Condition \eqref{eq constant cond} holds. Let $\xi : [e^{e}, \infty) \to (0, \infty)$ be a nondecreasing continuous function. Then $\xi \in LLC_{\infty}(M)$ if and only if
	\[
	\frac{\xi(t)}{t^H} \text{ is bounded and } 
	I_{\infty}(\xi) := \int_{e^{e}}^{\infty} \left( \frac{\xi(t)}{t^H} \right)^{-\frac{2}{2\gamma+3}} \varphi\!\left( \frac{\xi(t)}{t^H} \right) \frac{dt}{t} < +\infty.
	\]
\end{theorem}


\subsection{Proof of Theorems \ref{thoLLC0} and~\ref{thoLLCinfty}}

\subsubsection{A maximal inequality} 
\begin{proposition}\label{onemax1}
	Assume that Condition \eqref{eq constant cond} holds. Then  there exists a constant $c_{5,2} > 0$ such that,  for every $0 < t < u$ and $\theta, \eta > 0$,
\begin{equation}\label{onemax}
\mathbb{P}\left\{ M(t) \le \theta t^H,\; M(u) \le \eta \right\}
\le 2\varphi(\theta) \exp\left(-c_{5,2} (u-t) u^{\frac{2\alpha + 2\beta}{2\gamma+3}} \eta^{-\frac{2}{2\gamma+3}}\right).
\end{equation}

\begin{proof}
	If $$(u-t) u^{\frac{2\alpha + 2\beta}{2\gamma+3}} \eta^{-\frac{2}{2\gamma+3}} \le 2,$$ then \eqref{onemax} holds with $c_{5,2} = (\log 2)/2$. Hence  it remains to consider the case 
    \begin{align}\label{eq u 2}
	(u-t) u^{\frac{2\alpha + 2\beta}{2\gamma+3}} \eta^{-\frac{2}{2\gamma+3}}  > 2.
    \end{align}
    We divide the remainder of the proof into three steps.

\noindent	\textbf{Step 1.}  Starting from $t_0:=t$,  define  a sequence $\{t_n:\, n\ge 0\}$ recursively by 
	\begin{equation}\label{tnn-1}
		t_n - t_n^{-\frac{\alpha + \beta}{\lambda}} \eta^{\frac{1}{\lambda}} = t_{n-1}.
	\end{equation} 

    The sequence $\{t_n\}$ is strictly increasing and diverges to infinity. Indeed, if $t_n\to L<\infty$, then 
	$$
t_n-t_{n-1}
\longrightarrow
L^{-\frac{\alpha+\beta}{\lambda}}
\eta^{\frac{1}{\lambda}}\neq 
0,
$$
which contradicts the convergence of $\{t_n\}_{n\geq0}$.

    For each integer $k\ge0$,  define  the event
	$$
	A_k := \left\{M(t) \le \theta t^H\right\} \cap \left\{M(t_k) \le \eta\right\}.
	$$
Since $t_0=t<u$ and $t_k\to\infty$ as $k\to\infty$, there exists a unique integer $k_0\geq0$ such that
$$
t_{k_0}\leq u<t_{k_0+1}.
$$
Since $M$ is nondecreasing and $t_{k_0}\leq u$, we have
\begin{equation}\label{ak0-sup}
A_{k_0} \supseteq \left\{M(t) \le \theta t^H\right\} \cap \left\{M(u) \le \eta\right\}.
\end{equation}
It remains  to  prove that there  exist constants   $\rho\in(0,1)$ and $c_{5,3}>0$ such that
\begin{equation}\label{pak}
\mathbb{P}\{ A_k \} \le \varphi(\theta)\rho^k, \qquad k\ge 0.
\end{equation}
and  
\begin{equation}\label{k0-bound}
k_0 > c_{5,3}(u-t)u^{\frac{\alpha+\beta}{\lambda}}\eta^{-\frac{1}{\lambda}}.
\end{equation}
     Therefore, combining \eqref{ak0-sup}, \eqref{pak}, and \eqref{k0-bound}, we obtain \eqref{onemax}.

	\textbf{Step 2.} We  prove \eqref{pak} by induction on $k$.  The case $k=0$ follows from  \eqref{Eq:sb31}.  
	For the induction step,  observe that
	\[
	A_{k+1} \subseteq A_k \cap \bigl\{ |X(t_{k+1}) - X(t_k)| \le 2\eta \bigr\}.
	\]
	By \eqref{xt-xs}, we have 
	\begin{align*}
		X(t_{k+1}) - X(t_k) 
		&= \int_0^{t_k} r^\alpha \left( \int_{t_k}^{t_{k+1}} y^\beta (y - r)^\gamma dy \right) dW_r \\
		&\quad + \int_{t_k}^{t_{k+1}} r^\alpha \left( \int_r^{t_{k+1}} y^\beta (y - r)^\gamma dy \right)dW_r\\
		&= :J_1 + J_2.
	\end{align*}
	
By the lower bound for $I_2$ established in the proof of Lemma~\ref{lem mom} and \eqref{tnn-1}, we get
	$$\operatorname{Var}(J_2) \ge c_{2,1} t_{k+1}^{2\alpha+2\beta}|t_{k+1} - t_k|^{2\gamma+3} = c_{2,1}{\eta^2},$$ 
	which implies that 
\[
\mathbb{P}\{ |J_2| \le 2\eta \} \le \Phi(2/\sqrt{c_{2,1}}) - \Phi(-2/\sqrt{c_{2,1}})=:\rho.
\]
By Anderson's inequality \cite{Anderson55} and the independence of $J_2$ and $\sigma\{W(r): r \le t_k\}$, we obtain
\[
\begin{aligned}
\mathbb{P}\{ A_{k+1} \}
&\le \mathbb{E}\left[ \mathbb{P} \left\{ A_k \cap \{ |J_1+J_2| \leq 2\eta \} \mid \sigma\{W(r): r\le t_k\} \right\} \right] \label{eq:prob_bound}\\
&\le \mathbb{E}\Bigl[ \mathbf{1}_{A_k} \cdot \mathbb{P}\bigl\{ |J_1+J_2|\le 2\eta \mid \sigma\{W(r): r\le t_k\} \bigr\} \Bigr] \\
&\le \mathbb{P}\{ A_k \} \cdot \mathbb{P}\{ |J_2|\le 2\eta \} \\
&\le \mathbb{P}\{ A_k \} \rho.
\end{aligned}
\]
	Thus, \eqref{pak}  follows  by induction. 
    
\noindent	\textbf{Step 3.}   We now prove \eqref{k0-bound}. 
We first  claim that   
	\begin{equation}\label{tk0-bound}
		t_{k_0}=t_{k_0+1}-t_{k_0+1}^{-\frac{\alpha+\beta}{\lambda}}\eta^{\frac 1 \lambda}
		>u-u^{-\frac{\alpha+\beta}{\lambda}}\eta^{\frac 1 \lambda}.
	\end{equation}
    To this end,   define $$f(x):=x-x^{-\frac{\alpha+\beta}{\lambda}}\eta^{\frac1 \lambda}, \ \ x>0.$$  
  If $\alpha+\beta\ge 0$, then $f'(x)>0$ and $f$ is strictly increasing. If $\alpha+\beta<0$, then by \eqref{eq u 2}, $u^{H/\lambda}>2\eta^{1/\lambda}$. Consequently, for all $x\ge u$, we have
	\[
f'(x)=1+\frac{\alpha+\beta}{\lambda}x^{-\frac H \lambda}\eta^{\frac{2}{2\gamma+3}}>1+\frac{\alpha+\beta}{2\lambda}>\frac12.
	\] 
 Thus, $f$ is increasing on $[u,\infty)$.  Since $t_{k_0+1}>u$,  the recurrence relation  \eqref{tnn-1} gives  \eqref{tk0-bound}. Moreover,  by \eqref{eq u 2}, we have
\[
u^{-\frac{\alpha+\beta}{\lambda}}\eta^{\frac{1}{\lambda}}<\frac{u-t}{2}.
\]
Combining this estimate  with \eqref{tk0-bound}, we obtain
\begin{equation}\label{tk0-half}
t_{k_0}>u-u^{-\frac{\alpha+\beta}{\lambda}}\eta^{\frac{1}{\lambda}}
>\frac{u+t}{2}.
\end{equation} 

We now   derive the required lower bound for $k_0$. For
 $\alpha+\beta\le 0$, \eqref{tnn-1} yields
\[
t_{k_0}-t
=\sum_{n=1}^{k_0}(t_n-t_{n-1})
=\eta^{\frac{1}{\lambda}}\sum_{n=1}^{k_0}t_n^{-\frac{\alpha+\beta}{\lambda}}
\le k_0\eta^{\frac{1}{\lambda}}u^{-\frac{\alpha+\beta}{\lambda}}.
\]
Combining this with  \eqref{tk0-half} yields
\begin{equation}\label{k0-bound-case1}
k_0
> \frac12 (u-t)u^{\frac{\alpha+\beta}{\lambda}}\eta^{-\frac{1}{\lambda}}.
\end{equation}

For $\alpha+\beta>0$, \eqref{tnn-1} gives,  for every $1\le n\le k_0$,
$$
\eta^{\frac{1}{\lambda}}
= t_n^{\frac{\alpha+\beta}{\lambda}}(t_n-t_{n-1})
\ge \int_{t_{n-1}}^{t_n} x^{\frac{\alpha+\beta}{\lambda}}dx.
$$
Summing over $n=1,\ldots,k_0$  and applying \eqref{tk0-half},  we obtain
\begin{equation}\label{k0-bound-case2}
\begin{split}
k_0
>  \,  \eta^{-\frac{1}{\lambda}}\int_t^{\frac{t+u}{2}}x^{\frac{\alpha+\beta}{\lambda}}dx =\,  \eta^{-\frac{1}{\lambda}} u^{\frac{\alpha+\beta}{\lambda}+1}
\int_{\frac{t}{u}}^{\frac{t/u+1}{2}}
y^{\frac{\alpha+\beta}{\lambda}}\,dy
\ge  \, 
c_{5,4} (u-t)u^{\frac{\alpha+\beta}{\lambda}}\eta^{-\frac{1}{\lambda}},
\end{split}
\end{equation}
where 
\[
c_{5,4}
:=
\inf_{0\le s<1}
\frac{1}{1-s}
\int_s^{\frac{1+s}{2}}y^{\frac{\alpha+\beta}{\lambda}}dy.
\]
The function inside the infimum is positive and continuous on $[0,1)$ and extends continuously to $s=1$ with limit $1/2$. Therefore, $c_{5,4}>0$. 

Combining \eqref{k0-bound-case1} and \eqref{k0-bound-case2}, we obtain  \eqref{k0-bound}. The proof is complete. 
\end{proof}
\end{proposition}

\subsubsection{Small-ball functions}
Define
\begin{equation*}\label{eq:psi_def}
\psi(\varepsilon):=-\log\varphi(\varepsilon).
\end{equation*}
Then $\psi$ is nonnegative and    non-increasing.  Moreover, by Borell’s theorem  \cite{borell1974convex}, the function $\psi$ is convex. Consequently,  its right derivative $\psi'_+(\varepsilon)$ is non-decreasing and non-positive.  Hence,  $|\psi'_+|$ is non-increasing.

By Theorem \ref{lem small0t} with $t=1$, there exists a constant $c_{5,5} \ge 1$ such that for all $\varepsilon<1$,
\begin{equation}
	c_{5,5}^{-1} \varepsilon^{-\frac{2}{2\gamma+3}} \le \psi(\varepsilon) \le c_{5,5} \varepsilon^{-\frac{2}{2\gamma+3}}. \label{eq:psi_bound}
\end{equation}

The following lemmas give further properties of \(\varphi\) and \(\psi\), analogous to those for FBM in \cite[Section 2]{talagrand1996lower}. We omit the proofs, since they follow the same arguments as in \cite[Section 2]{talagrand1996lower}.
\begin{lemma}\label{convex1}Assume that Condition \eqref{eq constant cond} holds. 
There exists a constant $K_{5,1}$ such that for all $\varepsilon\in\left(0, \, K_{5,1}^{-1}\right)$,
\begin{equation}\label{convex}
-K_{5,1}\varepsilon^{-\frac{2}{2\gamma+3}-1}
\le \psi'_+(\varepsilon)
\le -K_{5,1}^{-1}\varepsilon^{-\frac{2}{2\gamma+3}-1}.
\end{equation}
\end{lemma}

\begin{lemma}\label{convex2}Assume that Condition \eqref{eq constant cond} holds. There exists a constant $K_{5,2}$ such that for all $\varepsilon \in \left(0, 2/K_{5,1}\right)$ and  $\theta >\varepsilon / 2 $,
	\begin{equation*}
		\exp\left(-K_{5,2} |\theta - \varepsilon|{\varepsilon^{-\frac{2}{2\gamma+3}-1}}\right)
		\le \frac{\varphi(\theta)}{\varphi(\varepsilon)}
		\le \exp\left(K_{5,2} |\theta -\varepsilon |\varepsilon^{-\frac{2}{2\gamma+3}-1}\right).
	\end{equation*}
\end{lemma}

\begin{lemma}\label{convex3}
Assume that Condition \eqref{eq constant cond} holds. For all $\varepsilon <  \left[(\gamma+\frac32)/K_{5,1}\right]^{\gamma+\frac32}$, the function $\varepsilon^{-\frac{2}{2\gamma+3}} \varphi(\varepsilon)$ is increasing.
\end{lemma}

\subsubsection{Another maximal inequality}
\begin{proposition}\label{twomax2}
Assume that Condition \eqref{eq constant cond} holds, and set 
$$\tau = \min \left\{ \frac{H}{6},\; \frac{1}{4}\bigl(\alpha+\frac12 \bigr) \right\}.$$
Then there exist positive constants $c_{5,6}, c_{5,7}$, and $c_{5,8}$ such that, for every  $u>t>0$, $\theta>0$, and  $\eta\in (0, c_{5,6}^{-1})$,
\begin{equation}\label{twomax}
\begin{split}
&\mathbb{P} \Bigl\{ \left\{ M(t) \leq \theta t^H \right\} \cap \left\{ M(u) \leq \eta u^H \right\} \Bigr\}\\
&\le c_{5,6}\exp \left[ -c_{5,7}\left( \frac{u}{t} \right)^{\tau} \right]
+ \varphi(\theta) \varphi(\eta) \exp\left (c_{5,8} \left( \frac{u}{t} \right)^{-\tau}  \eta^{-\frac{2}{2\gamma+3}-1}\right). 
\end{split}
\end{equation}
\end{proposition}
\begin{proof}    It suffices to consider the case where $u/t$ is sufficiently large. When $u/t$ is bounded, the desired estimate follows from the trivial bound after enlarging the constants if necessary.

		Set $v := \sqrt{u t}$.   Write the kernel in  \eqref{eq GVP} as 
	\begin{equation*}\label{kernel}
		K(s, r):= r^\alpha \left( \int_r^s y^\beta (y-r)^\gamma dy \right), \quad \, 0\le r\le  s. 
	\end{equation*}
	Consider the decomposition $X=X_1+X_2$, where, for $s\ge0$, 
\begin{equation*}\label{demo}
X_1(s):=\int_{0}^{v\wedge s} K(s,r)dW_r, \qquad
X_2(s):=\int_{v\wedge s}^{s} K(s,r)dW_r.
\end{equation*}
Then $X_1$ and $X_2$ are independent. Moreover, $X_2(s)=0$ for every $0\leq s\leq v$. Since $t<v$, it follows that
\begin{equation}\label{sup-decomp}
\sup_{0\le s\le t}|X(s)|
= \sup_{0\le s\le t}|X_1(s)+X_2(s)|
= \sup_{0\le s\le t}|X_1(s)|.
\end{equation}
For all $\delta>0$, we have
\begin{equation}\label{demomax}
\begin{aligned}
&\mathbb{P}\left\{ M(t)\le \theta t^H,\; M(u)\le \eta u^H \right\}\\
=& \, \mathbb{P}\left\{ \sup_{0\le s\le t}|X(s)|\le \theta t^H,\;
         \sup_{0\le s\le u}|X(s)|\le \eta u^H \right\}\\
\le &\, \mathbb{P}\left\{ \sup_{0\le s\le t}|X_1(s)|\le \theta t^H,\;
         \sup_{0\le s\le u}|X_2(s)|\le (\eta+\delta)u^H \right\}
         +\mathbb{P}\left\{ \sup_{0\le s\le u}|X_1(s)|\ge \delta u^H \right\}.
\end{aligned}
\end{equation}
By the independence of $X_1$ and $X_2$, we have 
\begin{equation}\label{demomax1}
\begin{aligned}
&\mathbb{P}\left\{ \sup_{0\le s\le t}|X_1(s)|\le \theta t^H,\;
         \sup_{0\le s\le u}|X_2(s)|\le (\eta+\delta)u^H \right\}\\
=&\,\mathbb{P}\left\{ \sup_{0\le s\le t}|X_1(s)|\le \theta t^H \right\}
 \cdot
 \mathbb{P}\left\{ \sup_{0\le s\le u}|X_2(s)|\le (\eta+\delta)u^H \right\}.
\end{aligned}
\end{equation}
Note that
\begin{equation}\label{demomax3}
\begin{aligned}
&\mathbb{P}\left\{ \sup_{0\le s\le u}|X_2(s)|\le (\eta+\delta)u^H \right\}\\
\le&\, \mathbb{P}\left\{ \sup_{0\le s\le u}|X(s)|\le (\eta+2\delta)u^H \right\}
+\mathbb{P}\left\{ \sup_{0\le s\le u}|X_1(s)|\ge \delta u^H \right\}.
\end{aligned}
\end{equation}
From \eqref{sup-decomp}--\eqref{demomax3}, we obtain
\begin{equation}\label{zdemomax}
\begin{aligned}
&\mathbb{P}\left\{ M(t)\le \theta t^H,\; M(u)\le \eta u^H \right\}\\
\le &\,\varphi(\theta)\,\varphi(\eta+2\delta)
+2\,\mathbb{P}\left\{ \sup_{0\le s\le u}|X_1(s)|\ge \delta u^H \right\}.
\end{aligned}
\end{equation}

    By the convexity of $\psi$ and \eqref{convex}, we have
\[
\psi(\eta+2\delta)
\ge \psi(\eta) + 2\delta\psi'_+(\eta)
\ge \psi(\eta) - 2\delta K_{5,1}\eta^{-\frac{2}{2\gamma+3}-1}.
\]
Since $\varphi=e^{-\psi}$, it follows that 
\[
\varphi(\eta+2\delta)
\le \varphi(\eta)\exp\left(2\delta K_{5,1}\eta^{-\frac{2}{2\gamma+3}-1}\right).
\] 
Choosing $\delta=(t/u)^\tau$, we have
\begin{equation}\label{Eq-dudley1}
\varphi(\theta)\,\varphi(\eta+2\delta)
\le \varphi(\theta)\varphi(\eta)
\exp\left(2K_{5,1}\left(u/t \right)^{-\tau}\eta^{-\frac{2}{2\gamma+3}-1}\right).
\end{equation} 

It remains to estimate the second term on the right-hand side of \eqref{zdemomax}. 
We claim that there exist    constants $c_{5,6}, c_{5,7}>0$ such that
\begin{equation}\label{Eq-dudley2}
2\mathbb P\left\{
\sup_{0\leq s\leq u}|X_1(s)|
\geq
\delta u^H
\right\}
\leq
c_{5,6}\exp\left\{
-c_{5,7}\left(\frac{u}{t}\right)^\tau
\right\}.
\end{equation}
Combining \eqref{zdemomax}, \eqref{Eq-dudley1}, and \eqref{Eq-dudley2}, we obtain \eqref{twomax}.  Next, we establish \eqref{Eq-dudley2}  using Lemma \ref{lem:variance_X1} below, whose proof is postponed.

\begin{lemma}\label{lem:variance_X1}
		Assume that  Condition \eqref{eq constant cond} holds. Then there exists a constant $c_{5,9} > 0$ such that, for every $0 <s \le u$,
\begin{equation}\label{eq:variance_X1_bound}
			\left(\mathbb E\left[X_1(s)^2\right]\right)^{1/2} \le c_{5,9}
			\begin{cases}
				v^H, & \beta+\gamma < -1,\\
                
		 v^H \bigl[1+\log(u/v)\bigr] & \beta+\gamma = -1,\\
				u^{\beta+\gamma+1} v^{\alpha+1/2}, & \beta+\gamma > -1.
			\end{cases}
		\end{equation} 
	\end{lemma}
By Lemmas  \ref{lem:covering} and \ref{lem:variance_X1}, it suffices to show that there exists a constant $c_{5,10}> 0$ such that
	\begin{equation}\label{eq_entropy_X1}
    \begin{split}
\mathcal{M}:= &\,\int_0^\infty \sqrt{\log N([0,u], d_{X_1}, \varepsilon)}d\varepsilon\\
		\le &\, 	c_{5,10}
		\begin{cases}
			\, v^H \log^{\frac12}\left(u/v\right), 
			&  \beta+\gamma < -1, \\[8pt]
			\, v^H \log^{\frac12}\left(u/v\right) \bigl[1+\log(u/v)\bigr]	, 
			&  \beta+\gamma = -1, \\[8pt]
			\, u^H \left(u/v\right)^{-(\alpha+\frac12)} \log^{\frac12}\left(u/v\right), 
			& \beta+\gamma > -1. 
		\end{cases}
        \end{split}
	\end{equation}
    We treat   $\beta+\gamma>-1$; the other cases are analogous.  
    
  Define $$d_X(s,t):=\bigl(\mathbb{E}|X(s)-X(t)|^2\bigr)^{1/2}.$$
 Since $X_1$ and $X_2$ are independent and $X=X_1+X_2$, we have
\[
d_X(s,t)^2=d_{X_1}(s,t)^2+d_{X_2}(s,t)^2,
\]
and hence $d_{X_1}(s,t)\le d_X(s,t)$. By Lemma \ref{lem:covering}, it follows that
\[
N([0,u],d_{X_1}, \varepsilon)
\le N([0,u],d_X, \varepsilon)\le c_{3,1} u^{\frac{2H}{2\gamma+3}} \varepsilon^{-\frac{2}{2\gamma+3}}.
\]
	Combining this with \eqref{eq:variance_X1_bound} yields  
	\begin{equation*}
		\begin{aligned}
			&\int_{0}^{\infty} \sqrt{\log N([0, u], d_{X_1}, \varepsilon)} d\varepsilon \\
			 \le&\, \int\limits_{0}^{c_{5,9} u^{\beta+\gamma+1} v^{\alpha+1/2}} 
			\sqrt{\log\left(  c_{3,1} u^{\frac{2H}{2\gamma+3}} \varepsilon^{-\frac{2}{2\gamma+3}} \right)} d\varepsilon \\
			 = &\, \left(\gamma+\frac{3}{2}\right) u^H \int\limits_{c_{5,9}^{-2/(2\gamma+3)} (u/v)^{(2\alpha+1)/(2\gamma+3)}}^{\infty}
			\sqrt{\log( c_{3,1} x)} \, x^{-(\gamma+3/2)-1} dx \\
			 = &\, (2\gamma+3) c_{3,1}^{\gamma+3/2} u^H 
			\int\limits_{\sqrt{ \log\left(  c_{3,1} c_{5,9}^{-2/(2\gamma+3)} (u/v)^{(2\alpha+1)/(2\gamma+3)} \right) }}^{\infty}
			y^2 e^{-(\gamma+3/2) y^2} dy \\
			 \le &\, c_{5,11} u^H \left(\frac{u}{v}\right)^{-(\alpha+1/2)} 
			\sqrt{\log\left( \frac{u}{v} \right)},
		\end{aligned}
	\end{equation*}
	for some constant $c_{5,11} > 0$. Here the second and third equalities follow from  the changes of variables 
    $$\varepsilon = u^{H} x^{-(\gamma+\frac{3}{2})}\  \,\text{and } \ x = e^{y^2}/ c_{3,1}, $$respectively. In the final step, we use the elementary estimate that for each $b>0$ and all sufficiently large $a$,
\begin{equation*}
\int_a^\infty y^2 e^{-b y^2} dy \le \frac{3a}{4b} e^{-b a^2}.
\end{equation*}

From \eqref{eq_entropy_X1},  $\delta=(t/u)^\tau$, and $\tau = \min \left\{ \frac{H}{6},\; \frac{1}{4}\bigl(\alpha+\frac12 \bigr) \right\}$, there exists a constant $c_{5,12}>0$ such that
\[
\frac{\delta^2 u^{2H}}{K_4^2 \mathcal M^2} \ge c_{5,12}\left(\frac{u}{t}\right)^\tau.
\]
Applying Lemma~\ref{lem:dudley} then gives
\begin{equation}\label{Eq-dudley3}
\mathbb{P}\left\{ \sup_{0\le s\le u}|X_1(s)|\ge \delta u^H \right\}
\le \exp\!\left(-\frac{\delta^2 u^{2H}}{K_4^2 \mathcal{M}^2}\right)
\le \exp\!\left(-c_{5,12} \left(\frac{u}{t}\right)^{\tau}\right),
\end{equation}
which proves \eqref{Eq-dudley2}.
The proof of Proposition~\ref{twomax2} is complete.
\end{proof}

\begin{proof}[Proof of Lemma \ref{lem:variance_X1}]

 We apply \eqref{eq:J-estimate} to prove \eqref{eq:variance_X1_bound} in three cases.
 
\noindent\text{Case 1: $\beta+\gamma<-1$.} 
By \eqref{eq:J-estimate}, we get
\[
\mathbb E\left[X_1(s)^2\right]= \int_0^{v\wedge s} r^{2\alpha} \left( \int_r^s y^\beta (y-r)^\gamma dy \right)^2 dr
\le c_{4,4}^2 \int_0^{v\wedge s} r^{2\alpha+2\beta+2\gamma+2} dr 
\le \frac{c_{4,4}^2}{2H} v^{2H}.
\]

\noindent\text{Case 2:} $\beta+\gamma=-1$.  
By \eqref{eq:J-estimate},  the change of variables $r=vz$, and the elementary inequality $(x+y)^2\le 2x^2+2y^2$,  we have
\begin{equation*}\label{eq:X1norm-case3}
\begin{aligned}
\mathbb E\left[X_1(s)^2\right]
&\le c_{4,4}^2 \int_0^v r^{2\alpha}\left(1+\log\frac{u}{r}\right)^2 dr \\
&= c_{4,4}^2 v^{2\alpha+1} \int_0^1 z^{2\alpha}\left(1+\log\frac{u}{vz}\right)^2 dz \\
&\le c_{4,4}^2 v^{2\alpha+1} \left[
2\left(1+\log\frac{u}{v}\right)^2\int_0^1 z^{2\alpha}dz
+ 2\int_0^1 z^{2\alpha}\log^2\frac{1}{z}\,dz
\right] \\
&= c_{4,4}^2 v^{2\alpha+1} \left[
\frac{2}{2\alpha+1}\left(1+\log\frac{u}{v}\right)^2
+ \frac{4}{(2\alpha+1)^3}
\right] \\
&\le c_{5,13} v^{2H} \left(1+\log\frac{u}{v}\right)^2,
\end{aligned}
\end{equation*}
for some constant $c_{5,13}>0$. Here,  we used  the identity
\[
\int_0^1 z^{2\alpha}\log^2\frac1zdz
= \int_0^\infty w^2 e^{-(2\alpha+1)w}dw
= \frac{2}{(2\alpha+1)^3}.
\]

\noindent\text{Case 3:} $\beta+\gamma>-1$. By \eqref{eq:J-estimate}, for every  $s\le u$, we have 
\begin{equation*}\label{X1-norm-case1}
\begin{aligned}
\mathbb E\left[X_1(s)^2\right]
\le c_{4,4}^2 s^{2\beta+2\gamma+2}\int_0^{v\wedge s} r^{2\alpha}dr
\le \frac{c_{4,4}^2}{2\alpha+1} u^{2\beta+2\gamma+2} v^{2\alpha+1}.
\end{aligned}
\end{equation*}
	The  proof is complete. 
\end{proof}
\begin{proof}[Proof of Theorems \ref{thoLLC0} and~\ref{thoLLCinfty}]
We follow the arguments in \cite[Sections 4 and 5]{wang2022lower};
see also \cite[Section 3]{LWW}, where the same reduction is used.
The two maximal inequalities in Propositions~\ref{onemax1} and~\ref{twomax2}, together with the small-ball function properties in Lemmas~\ref{convex1}-\ref{convex3}, are exactly the ingredients needed to apply the arguments in the cited papers.

Consequently, the recursive discretization of the time parameter, the
comparison between the resulting series and the integrals defining
$I_0(\xi)$ and $I_\infty(\xi)$, and the associated Borel-Cantelli and
second-moment arguments apply without substantive change, with the
small ball exponent $2/(2\gamma+3)$ in place of the corresponding exponent
in the cited papers. The additional condition
$$
\xi(2t)\leq c_{5,1}\xi(t)
$$
is used only in the necessity part at the origin, exactly as in
\cite[Theorem 2.1(b)]{wang2022lower}. To avoid repeating essentially identical
arguments, we omit the remaining details.
\end{proof}

\subsection{Proof of  Theorem~\ref{thm CLTL0infinity}}
The following lemma provides a weaker form of Theorem~\ref{thm CLTL0infinity}, in which \(\kappa_6'\) and \(\kappa_7'\) may be \(0\) or \(\infty\). Thus it remains to prove that these constants are positive and finite.

\begin{lemma}\label{lem:zeroone-origin-infinity}
Assume Condition \eqref{eq constant cond} holds. There exist constants $\kappa_6',\kappa_7'\in[0,\infty]$ such that
$$
\liminf_{t\to0^+}
\sup_{0\le s\le t}
\frac{(\log\log t^{-1})^{\gamma+\frac32}}
{t^H}|X(s)|
=\kappa_6',
\qquad\text{a.s.,}
$$
and
$$
\liminf_{t\to\infty}
\sup_{0\le s\le t}
\frac{(\log\log t)^{\gamma+\frac32}}
{t^H }|X(s)|=\kappa_7',
\qquad\text{a.s.}
$$
\end{lemma}

\begin{proof}
Lemma~\ref{key1} gives \(r(t)\to0\) as \(|t|\to\infty\). Consequently, by Maruyama's theorem \cite{maruyama1949}, \(U\) is strongly mixing. Therefore, for every $h\neq0$, the shift $\theta_h$ is ergodic. Under the Lamperti correspondence, $\theta_h$ is conjugate to
$$
(S_aX)(t):=a^{-H}X(at),
\qquad a=e^h,
$$
so $S_a$ is ergodic for every $a>0$ with $a\neq1$.

Denote the two normalized liminf quantities by $L_0(X)$ and $L_\infty(X)$. Since
$$
\sup_{0\leq s\leq t}|(S_aX)(s)|=
a^{-H}\sup_{0\leq s\leq at}|X(s)|,
$$
and
$$
\frac{\log\log(u/a)}{\log\log u}\to1
\quad\text{as }u\to\infty,
\qquad
\frac{\log\log(a/u)}{\log\log(1/u)}\to1
\quad\text{as }u\to0^+,
$$
we have
$$
L_0(S_aX)=L_0(X),
\qquad
L_\infty(S_aX)=L_\infty(X)
$$
almost surely.  Since, for every $c\geq0$, the events
$$
{L_0(X)\leq c}
\qquad\text{and}\qquad
{L_\infty(X)\leq c}
$$
are invariant under any fixed ergodic transformation $S_a$ with $a\neq1$, the zero-one law implies that both \(L_0(X)\) and \(L_\infty(X)\) are almost surely deterministic constants. 

The proof is complete.
\end{proof}

\begin{proof}[Proof of Theorem~\ref{thm CLTL0infinity}] We now use Theorems~\ref{thoLLC0} and~\ref{thoLLCinfty} to show that \(\kappa_6'\) and \(\kappa_7'\) are positive and finite.
For \(c>0\), define
\[
\xi_{0,c}(t)
:=
ct^H(\log\log t^{-1})^{-(\gamma+\frac32)},
\qquad 0<t\le e^{-e}.
\]
This function is continuous and nondecreasing, and it satisfies the
doubling condition in Theorem~\ref{thoLLC0}(b). By the change of variables
\(v=\log(t^{-1})\),
\[
I_0(\xi_{0,c})
=
c^{-(\gamma+\frac32)^{-1}}
\int_e^\infty
(\log v)
\varphi\bigl(c(\log v)^{-(\gamma+\frac32)}\bigr)\,dv.
\]
The two-sided small ball estimate in Theorem \ref{lem small0t} gives, for all
sufficiently large \(v\),
\[
v^{-\kappa_1c^{-(\gamma+\frac32)^{-1}}}
\le
\varphi\bigl(c(\log v)^{-(\gamma+\frac32)}\bigr)
\le
v^{-\kappa_2c^{-(\gamma+\frac32)^{-1}}}.
\]
Consequently,
\[
I_0(\xi_{0,c})<\infty
\quad\text{if }c<\kappa_2^{\gamma+\frac32},
\]
whereas
\[
I_0(\xi_{0,c})=\infty
\quad\text{if }c>\kappa_1^{\gamma+\frac32}.
\]

If \(c<\kappa_2^{\gamma+\frac32}\), Theorem~\ref{thoLLC0}(a) yields
\(\xi_{0,c}\in\mathrm{LLC}_0(M)\). Hence
\[
M(t)\ge
ct^H(\log\log t^{-1})^{-(\gamma+\frac32)}
\]
for all sufficiently small \(t\), almost surely, and therefore
$\kappa'_6\ge c$.
Letting \(c\uparrow\kappa_2^{\gamma+\frac32}\) gives
\[
\kappa'_6\ge\kappa_2^{\gamma+\frac32}.
\]

If \(c>\kappa_1^{\gamma+\frac32}\), then \(I_0(\xi_{0,c})=\infty\).
By the contrapositive of Theorem~\ref{thoLLC0}(b),
\[
\xi_{0,c}\notin\mathrm{LLC}_0(M).
\]
If \(\kappa'_6>c\), then the definition of the liminf would imply
\[
M(t)\ge
ct^H(\log\log t^{-1})^{-(\gamma+\frac32)}
\]
for all sufficiently small \(t\), almost surely. This would imply \(\xi_{0,c}\in\mathrm{LLC}_0(M)\), contradicting the conclusion above. Thus
\(\kappa'_6\le c\). Letting \(c\downarrow\kappa_1^{\gamma+\frac32}\), we obtain
\[
\kappa_2^{\gamma+\frac32}
\le
\kappa'_6
\le
\kappa_1^{\gamma+\frac32}.
\]

The argument for \(\kappa_7'\) is identical. 
Choose $T\ge e^e$ sufficiently large so that
$$
t\longmapsto
ct^H(\log\log t)^{-\lambda}
$$
is nondecreasing on $[T,\infty)$.
 Set
\[
\xi_{\infty,c}(t):=ct^H(\log\log t)^{-\lambda}, \ \ \ t\ge T.
\]
 Then, by Theorem~\ref{thoLLCinfty}, we have
\[
\kappa_2^\lambda\le \kappa_7'\le \kappa_1^\lambda.
\]

Therefore, since \(\kappa_1,\kappa_2\in(0,\infty)\), both \(\kappa_6'\) and \(\kappa_7'\) are positive and finite. This proves Theorem~\ref{thm CLTL0infinity}.
\end{proof}

\section{Appendix}
	\begin{lemma}\cite[Lemma 4]{MS22a}\label{lem 4 MSa}
    Let $\beta\in \mathbb R$ and $\gamma>-1$.  Then, for every fixed $t>0$, the following asymptotic relations hold as $s\to0^+$.
		\begin{itemize}
			\item[(i)] When $\beta+\gamma<-1$,
			\[
			\int_s^t u^{\beta}(u-s)^{\gamma}du\sim s^{\beta+\gamma+1} \mathbf B(\gamma+1, -\beta-\gamma-1).
			\]
			\item[(ii)] When $\beta+\gamma=-1$,
			\[
			\int_s^t u^{\beta}(u-s)^{\gamma}du\sim \log \left(\frac{t}{s}\right).
			\]
			\item[(iii)] When $\beta+\gamma>-1$,
			\[
			\int_s^t u^{\beta}(u-s)^{\gamma}du\rightarrow \frac{t^{\beta+\gamma+1}}{\beta+\gamma+1}.
			\]
		\end{itemize}
	\end{lemma}
	
Let $\{Z(t):\, t \in S\}$ be a separable, real-valued, centered Gaussian process indexed by a bounded set $S$ with the canonical metric 
$$d_Z(s, t) = (\mathbb E |Z(s) - Z(t)|^2)^{\frac{1}{2}}.$$ Let $N(S,d_Z,\varepsilon)$ denote 
the smallest number of $d_Z$-balls of radius $\varepsilon$ needed to cover $S$.
	\begin{lemma}\cite[p. 257, (7.11)--(7.13)]{Ledoux}\label{Lem:Ta93}
		  If there is a decreasing function $f: (0, \delta] \to (0, \infty)$ such that $$N(S,d_Z,\varepsilon) \le f(\varepsilon), \ \text{for all } \varepsilon \in (0, \delta],$$  and there are constants $K_2 \ge K_1 > 1$ such that
		\begin{equation*}\label{Eq:Covering}
			K_1 f (\varepsilon) \le f (\varepsilon/2) \le K_2 f (\varepsilon)
		\end{equation*}
		for all $\varepsilon \in (0, \delta]$, then there is a constant $K$ depending  on $K_1$ and $K_2$  such that, for all $u \in (0, \delta)$,
		\begin{equation*}\label{Eq:SB1}
\mathbb{P}\left\{ \sup_{s, t \in S} |Z(s) - Z(t)| \le u \right\} \ge \exp\left(-K f(u) \right).
\end{equation*}
\end{lemma}
	
	\begin{lemma}\cite[Lemma 1.1, p. 138]{JM1978}\label{lem: Fern}
		Let $\{Z(t):\ t\ge0\}$ be a separable, centered, real-valued Gaussian process. Assume that 
		\[
		\mathbb E\left[\left(Z(t+h)-Z(t)\right)^2\right]\le \varphi(h)^2, \ \ \ t>0, \  h>0,
		\]
		for some continuous nondecreasing function $\varphi$ with $\varphi(0)=0$. 
         Then, for every integer $K_3>1$, all $t,x>0$, and every sequence of positive numbers ${\theta(p)}_{p\in\mathbb N}$, we have
\begin{align*}
& \mathbb{P}\left\{ \sup_{0\le s\le t}|Z(s)-Z(0)|>x\varphi(t)+\sum_{p=1}^{\infty} \theta(p)\varphi\left(tK_3^{-2^p} \right)  \right\} \\
&\le K_3^2 e^{-x^2/2}+\sum_{p=1}^{\infty} K_3^{2^{p+1}}e^{-\theta(p)^2/2}.
\end{align*}
\end{lemma}
	
	\begin{lemma}\cite[Lemma 2.1]{talagrand1996lower}\label{lem:dudley}
		There exists a universal constant $K_4 > 0$ such that for all $t_0 \in S$ and $x > 0$,
		\[
\mathbb{P} \left\{ \sup_{t \in S} |Z(t) - Z(t_0)| \geq K_4 x \int_{0}^{\infty} \sqrt{\log N(S, d_Z, \varepsilon)} \, d\varepsilon \right\} \leq \exp\left(-x^2\right).
\]
	\end{lemma}
	
	
	 \vskip0.8cm

 \noindent{\bf Author Contributions}:   All authors contributed equally to this work.

 \noindent{\bf Funding}:  The research of R. Wang is partially supported by the NSF of Hubei Province (Grant No. 2024AFB683).

 \noindent{\bf  Data Availability}:   
 Data Availability No datasets were generated or analyzed during the current study.

 \noindent {\bf Conflict of Interest}:   The authors declare that they have no conflict of interest.


\vskip0.9cm
	
	 
	 \begin{thebibliography}{99}


\bibitem{Anderson55}
Anderson, T.W.: The integral of a symmetric unimodal function over a symmetric convex set
and some probability inequalities. \textit{Proc. Am. Math. Soc.} \textbf{6}, 170--176 (1955)

\bibitem{BMRS}
Banna, O., Mishura, Y., Ralchenko, K., Shklyar, S.:
\textit{Fractional Brownian Motion: Approximations and Projections}. Wiley, Hoboken (2019)

\bibitem{borell1974convex}
Borell, C.: Convex measures on locally convex spaces. \textit{Ark. Mat.} \textbf{12},
239--252 (1974)


\bibitem{elnouty2004}
El-Nouty, C.: Lower classes of integrated fractional Brownian motion.
\textit{Studia Sci. Math. Hungar.} \textbf{41}(1), 17--38 (2004)

\bibitem{elnouty2011}
El-Nouty, C.: Lower classes of the Riemann--Liouville process.
\textit{Bull. Sci. Math.} \textbf{135}(1), 113--123 (2011)

\bibitem{EO24}
El Omari, M.: On the Gaussian Volterra processes with power-type kernels.
\textit{Stoch. Models} \textbf{40}(1), 152--165 (2024)

\bibitem{TGM2022}
Ichiba, T., Pang, G., Taqqu, M.S.: Path properties of a generalized fractional Brownian motion.
\textit{J. Theor. Probab.} \textbf{35}, 550--574 (2022)

\bibitem{JM1978}
Jain, N.C., Marcus, M.B.: Continuity of sub-Gaussian processes. In: 
  \textit{Probability on Banach Spaces}, Adv. Probab. Related Topics, vol. 4,
pp. 81--196. Dekker, New York (1978)

\bibitem{Ledoux}
Ledoux, M.: Isoperimetry and Gaussian analysis.
In: \textit{Lectures on Probability Theory and Statistics, Saint-Flour 1994},
Lecture Notes in Math., vol. 1648, pp. 165--294. Springer, Berlin (1996)

 

\bibitem{li_shao_2001}
Li, W.V., Shao, Q.-M.: Gaussian processes: inequalities, small ball probabilities
and applications.
In: Rao, C.R., Shanbhag, D. (eds.) \textit{Stochastic Processes: Theory and Methods},
Handbook of Statistics, vol. 19, pp. 533--597. North-Holland, Amsterdam (2001)


\bibitem{Lo}
Lo\`eve, M.: \textit{Probability Theory I}. Springer, New York (1977)

\bibitem{LWW}
Lyu, M., Wang, M., Wang, R.: Lower classes and Chung's LILs of the fractional
integrated generalized fractional Brownian motion. \textit{Acta Math. Sci.}
\textbf{46}(3), 1518--1535 (2026)


\bibitem{maruyama1949}
Maruyama, G.: The harmonic analysis of stationary stochastic processes.
\textit{Mem. Fac. Sci. Kyushu Univ. Ser. A} \textbf{4}, 45--106 (1949)



\bibitem{Mis08}
Mishura, Y.: \textit{Stochastic Calculus for Fractional Brownian Motion and Related Processes}.
Springer, Berlin (2008)

\bibitem{MS22a}
Mishura, Y., Shklyar, S.: Gaussian Volterra processes with power-type kernels. Part I.
\textit{Mod. Stoch. Theory Appl.} \textbf{9}(3), 313--338 (2022)

\bibitem{MS22b}
Mishura, Y., Shklyar, S.: Gaussian Volterra processes with power-type kernels. Part II.
\textit{Mod. Stoch. Theory Appl.} \textbf{9}(4), 431--452 (2022)

\bibitem{MRS23}
Mishura, Y., Ralchenko, K., Shklyar, S.: Gaussian Volterra processes: asymptotic growth
and statistical estimation. \textit{Theory Probab. Math. Stat.} \textbf{108}, 149--167 (2023)

\bibitem{MR1995}
Monrad, D., Rootz\'en, H.: Small values of Gaussian processes and functional laws of the
iterated logarithm. \textit{Probab. Theory Relat. Fields} \textbf{101}, 173--192 (1995)

\bibitem{NVV99}
Norros, I., Valkeila, E., Virtamo, J.: An elementary approach to a Girsanov formula and
other analytical results on fractional Brownian motion. \textit{Bernoulli} \textbf{5}(4),
571--587 (1999)

\bibitem{revesz1990random}
R\'ev\'esz, P.: \textit{Random Walk in Random and Non-Random Environments}.
World Scientific, Teaneck (1990)

 


\bibitem{talagrand1996lower}
Talagrand, M.: Lower classes for fractional Brownian motion.
\textit{J. Theor. Probab.} \textbf{9}, 191--213 (1996)
 

\bibitem{TX2007}
Tudor, C.A., Xiao, Y.: Sample path properties of bifractional Brownian motion.
\textit{Bernoulli} \textbf{13}, 1023--1052 (2007)


\bibitem{wang2022lower}
Wang, R., Xiao, Y.: Lower functions and Chung's LILs of the generalized fractional Brownian
motion. \textit{J. Math. Anal. Appl.} \textbf{514}, 126320 (2022)

\bibitem{WX2022a}
Wang, R., Xiao, Y.: Exact uniform modulus of continuity and Chung's LIL for the generalized
fractional Brownian motion. \textit{J. Theor. Probab.} \textbf{35}(4), 2442--2479 (2022)

\bibitem{WSX}
Wang, W., Su, Z., Xiao, Y.: The moduli of non-differentiability for Gaussian random fields
with stationary increments. \textit{Bernoulli} \textbf{26}, 1410--1430 (2020)


 

\bibitem{Xiao2008}
Xiao, Y.:   Strong local nondeterminism and sample path properties of Gaussian random fields. In: {\it Asymptotic Theory in Probability and Statistics with Applications} (T.-L. Lai, Q.-M. Shao and L. Qian, eds.), pp. 136--176, Higher Education Press, Beijing (2007)

\end{thebibliography}
\end{document}